\documentclass[leqno]{amsart}

\usepackage{amssymb}
\usepackage{amsmath}

\usepackage{lineno,hyperref}

\usepackage{graphicx}

\usepackage{color}

\numberwithin{equation}{section}

\newcommand{\R}{\mathbb{R}}
\newcommand{\C}{\mathcal{C}}
\newcommand{\n}{\mathbf{n}}

\newcommand{\CL}{\mathcal{L}}
\newcommand{\CB}{\mathcal{B}}

\def\Xint#1{\mathchoice
    {\XXint\displaystyle\textstyle{#1}}%
     {\XXint\textstyle\scriptstyle{#1}}%
     {\XXint\scriptstyle\scriptscriptstyle{#1}}%
     {\XXint\scriptstyle\scriptscriptstyle{#1}}%
	\!\int}
\def\XXint#1#2#3{{\setbox0=\hbox{$#1{#2#3}{\int}$}
	\vcenter{\hbox{$#2#3$}}\kern-.5\wd0}}

\newtheorem{theorem}{Theorem}[section]

\theoremstyle{definition}
\newtheorem{definition}[theorem]{Definition}

\theoremstyle{definition}

\theoremstyle{plain}
\newtheorem{thm}{Theorem}[section]
\newtheorem{lemma}[thm]{Lemma}
\newtheorem{prop}[thm]{Proposition}
\newtheorem{cor}[thm]{Corollary}
\newtheorem{remark}[thm]{Remark}

\begin{document}

\title[Venttsel problems with discontinuous data]{Venttsel boundary value problems with discontinuous data}

\author[D.E. Apushkinskaya, A.I. Nazarov, D.K. Palagachev, L.G. Softova]{Darya E. Apushkinskaya\and Alexander I. Nazarov\and Dian K. Palagachev\and Lubomira G. Softova}

\address{Darya E. Apushkinskaya: Saarland University, Saarbr\"ucken, Germany; Peoples' Fri\-end\-ship University of Russia (RUDN University), Moscow and Chebyshev Laboratory, St. Petersburg State University, St. Petersburg, Russia}
\email{darya@math.uni-sb.de}

\address{Alexander I. Nazarov: PDMI RAS and St. Petersburg State University, St. Petersburg, Russia}
\email{al.il.nazarov@gmail.com}

\address{Dian K. Palagachev: Polytechnic University of Bari, Department of Mechanics, Mathematics and Management, 70125 Bari, Italy}
\email{dian.palagachev@poliba.it}

\address{Lubomira G. Softova: University of Salerno, Department of Mathematics, 84084 Salerno, Italy}
\email{lsoftova@unisa.it}



\begin{abstract}
We study linear and quasilinear Venttsel boundary value problems involving elliptic operators with discontinuous coefficients. On the base of the \textit{a~priori} estimates obtained, maximal regularity and strong solvability in Sobolev spaces are proved. 
\end{abstract}

\date{\today}
     
\maketitle

\hfill {\it Dedicated to Nina N. Uraltseva with admiration     }

\section{Introduction}\label{sec1}

The paper deals with \textit{discontinuous} Venttsel boundary value problems for linear and quasilinear  second-order elliptic equations. The discontinuity regards the coefficients of the differential operators acting inside and on the boundary of the underlying domain, and is expressed in terms of appurtenance of the principal coefficients to the class of functions with \textit{vanishing mean oscillation (VMO),} while 
optimal Lebesgue integrability requirements are imposed on the lower-order coefficients. We consider \textit{strong solutions} belonging to Sobolev spaces with optimal exponent ranges and develop $L^p$-regularity and  solvability theory for such problems.

The history of the Venttsel BVPs  goes back to the pioneering work \cite{V59} where, given an elliptic operator $\mathcal{L}$ in a bounded domain $\Omega\subset\mathbb{R}^n,$ A.D.~Venttsel found  general  boundary conditions, given in terms of a  second-order integro-differential operator,  which restrict $\mathcal{L}$ to an infinitesimal generator of a Markov process in $\Omega.$ From a probabilistic point of view, 
the Venttsel conditions\footnote{Notice that Venttsel conditions are known in the literature also as \textit{Wentzell} or \textit{Ventcel'} conditions. We mention also that Venttsel type conditions occur for all types of second-order equations -- elliptic, parabolic and hyperbolic; we discuss here only the elliptic case.} are the \textit{most general admissible boundary conditions}, where the differential terms
describe the phenomena of diffusion along the boundary, absorption,  reflection and viscosity, while the non-local integral terms represent jumps of the process along $\partial\Omega$ and inward jumps from $\partial\Omega$ into $\Omega$ (see also \cite{IW89} and \cite{W79}). 

We consider here the pure local case when the Venttsel conditions are given by a combination of second-order differential operator {\it along the boundary} and a full gradient term. Even in such settings, the Venttsel conditions
 include as particular cases the Dirichlet, Neumann, oblique derivative and mixed (Robin) boundary conditions.
 
The Venttsel problems describe many physical processes in media surrounded by a thin film, and appear in various branches of science, technology and industry: e.g. in water wave theory (\cite{Sh80}), electromagnetic and phase-transition phenomena (\cite{SV95} and \cite{HS74}), elasticity theory problems (\cite{MNaPr86}), engineering problems of hydraulic fracturing (\cite{CaMe71}), models of fluid diffusion (\cite{NaPi93}), as well as in some climate models or non-isothermal phase separation in a confined container (\cite{GGM08,GM09}), and in various aspects of financial mathematics (\cite[Chapter~8]{Shi99}). Some simple physical models leading to problems with Venttsel boundary conditions can be also found in  \cite{AN97}, \cite{LN98} and \cite{Gol06}. Moreover, the Venttsel problem is the simplest case of systems connected on manifolds of different dimensions (cf. \cite{NaPi93}), and it also provides an example of problems on stratified sets (see, for instance, \cite{Pe98}).

There is a vast literature devoted to linear and semilinear Venttsel problems 
in both local and non-local settings, see e.g. \cite{LuT91}, \cite{AN01}, \cite{GaSk01}, \cite{La02}, \cite{La03}, \cite{CF3GOR09}, \cite{NP15}, \cite{CLNV19} and the references therein. 
The study of \textit{quasilinear} problems with Venttsel boundary conditions was initiated  by Y.~Luo in \cite{Lu91} and continued later in a series of publications by D.E.~Apushkinskaya and A.I.~Nazarov. A detailed survey on the ``quasilinear'' results obtained up to 1999 can be found in \cite{AN00}. We mention also a series of papers dealing with \textit{two-phase quasilinear} Venttsel problems (\cite{AN00a,AN02}) where  the Venttsel condition is given on an interface separating the domain in two parts, and \textit{degenerate quasilinear}  Venttsel problems (\cite{AN97,AN98,N04}) where the thickness of the surrounding film can become zero on a subset of the boundary. 

Notice that all the results mentioned above   concern equations and boundary conditions with leading terms that depend at least \textit{continuously} on the independent variable $x.$

\smallskip

The first relevant $W^2_p$-theory of linear elliptic operators $a^{ij}(x)D_iD_j$ with \textit{discontinuous} coefficients is due to F.~Chiarenza, M.~Frasca and P.~Longo. In their pioneer  works \cite{CFL91, CFL93} the authors allowed discontinuity of $a^{ij}$s, taking these in the Sarason class \textit{VMO}, 
that contains only as a proper subset the space of uniformly continuous functions (see Section~\ref{sec2} for more details). It is proved in  \cite{CFL91, CFL93} that $a^{ij}\in \textit{VMO}$ ensures the validity of the Calder\'on--Zygmund property for \textit{any} $p>1.$ Namely, if $u\in W^2_q(\Omega)$ with $q\in(1,p)$ and $u=0$ on $\partial\Omega$ in the sense of traces, then $a^{ij}(x)D_iD_ju\in L^p(\Omega)$  implies $u\in W^2_p(\Omega).$ This is a crucial  point in the $L^p$-theory of PDEs allowing to extend the results for operators with continuous coefficients (cf. \cite{LU68,GT01}) to discontinuous ones.  
Moreover, since the  space $W^1_n(\Omega)$ and the fractional Sobolev space $W^{\vartheta}_{n/\vartheta}(\Omega),$ $\vartheta\in (0,1),$ are both contained in \textit{VMO},  the \textit{VMO}-discontinuity  of the coefficients makes these results more general then those obtained before (see \cite{CFL93} and \cite[Chapter~1]{MPS00} for more references).
The technique in \cite{CFL91, CFL93} is based on  an explicit representation formula for the  derivatives $D_iD_ju$ via Calder\'on--Zygmund singular integrals $\mathcal{K} f$ and their commutators $\C[a^{ij}, D_iD_j]=a^{ij}\mathcal{K}D_iD_j-\mathcal{K}(a^{ij}D_iD_j),$  
and the vanishing property of the \textit{VMO}-moduli of $a^{ij}$s permits to make the commutator norm small enough and hence to obtain an \textit{a~priori} estimate for the strong  solutions. Combining this estimate with a fixed point theorem arguments, the authors of \cite{CFL91, CFL93} obtain unique strong solvability of the Dirichlet problem
$$
a^{ij}(x)D_iD_ju=f(x)\quad \text{a.e. in}\ \Omega,\qquad
u=0\quad \text{on}\ \partial\Omega
$$
in $W^2_p(\Omega)$ for every $f \in L^p(\Omega)$ provided  $p\in(1,\infty)$.

Similar regularity and strong solvability theory have been developed in \cite{FaP96a} and \cite{MP98} for linear oblique derivative problem for  uniformly elliptic operators with \textit{VMO} principal coefficients.

Combining the results of \cite{CFL91, CFL93} with the Aleksandrov--Bakel'man maximum principle, suitable \textit{a~priori} gradient estimates have been obtained in \cite{P95} for the strong solutions to \textit{quasilinear} elliptic equations with \textit{VMO} principal coefficients. As consequence, $W^2_n$-solvability of the \textit{quasilinear} Dirichlet problem was proved in \cite{P95}. In \cite{FaP96a} similar results were obtained for the oblique derivative problem.

\smallskip

In the present paper we develop 
a strong solvability theory for linear and quasilinear elliptic Venttsel problems  with \textit{discontinuous}  coefficients. We deal with \textit{strong} solutions 
lying in the space $V_{p,q}(\Omega)=W^2_p(\Omega)\cap W^2_q(\partial\Omega)$ that satisfy the interior and the boundary equations almost everywhere.

For the linear case, the exponents $p$ and $q$ vary in the full range $(1,\infty),$ restricted only by a natural requirement (cf. \eqref{pq-cond}) ensuring trace compatibility. The principal coefficients of  both the domain and boundary operators are taken in  \textit{VMO}, while \textit{optimal}  integrability in Lebesgue or Orlicz spaces is assumed for the lower-order coefficients.  
The analytic core here is Theorem~\ref{apriori-estimate} that provides a coercive \textit{a~priori} estimate in $V_{p,q}(\Omega)$ for any solution to the linear Venttsel problem. The proof relies on the results of \cite{CFL91,CFL93}, fine interpolation between various Sobolev spaces, depending on the admissible combinations of $p$ and $q,$ and $W^2_p$-bounds for suitable extension operators (see Theorem~\ref{extension-theorem}). The coercive estimate implies the \textit{Fredholm property}, which, in turn, provides the 
 \textit{improving of integrability} property for the linear Venttsel problem (Theorem~\ref{regular}). Finally, a comparison principle of  \textit{Aleksandrov--Bakel'man} type (Theorem~\ref{global-max-principle}) allows to derive unique strong solvability for all admissible values of $p$ and $q.$


The natural functional framework for the solutions of the \textit{quasilinear} Venttsel problem is the space $V_{n,n-1}(\Omega).$ Due to technical difficulties, we restrict ourselves to the case when the principal coefficients both of the equation and the boundary condition are \textit{independent} of the gradient
of solution.
However, these depend on the solution itself and 
exhibit \textit{discontinuities} in the independent variable $x,$ measured in terms of
\textit{VMO}. The lower-order terms support quadratic gradient growth  and may have unbounded singularities in $x$ with a proper control of the Lebesgue or Orlicz integrability. The existence approach relies on the Leray--Schauder fixed point theorem that reduces the solvability to suitable \textit{a~priori} estimates for all possible solutions to a  family of Venttsel problems. In our case these are bounds for the $L^{2n}(\Omega)$-norm of the \textit{full} gradient and the $L^{2(n-1)}(\partial\Omega)$-norm of the \textit{tangential} gradient (Theorem~\ref{Du-est}). To get such estimates, we adapt to the specific Venttsel situation a homotopy-type machinery of Amann and Crandall~\cite{AC78} which reveals very useful when dealing with discontinuous quasilinear operators.
This approach requires also estimates for the $L^\infty$ and the H\"older norms of the solution. The first is obtained in Lemma~\ref{M0}, while the second one 
follows for free from \cite{AN95a}. As a result, strong solvability of the quasilinear Venttsel problem (Theorem~\ref{quasilinear-existence}) does follow.

\smallskip
The  paper is organized as follows. In Section~\ref{sec2} we provide the necessary notation, collect the basic facts about \textit{VMO} spaces, and prove an auxiliary result about extension operators in Sobolev spaces. In Sections~\ref{sec3} and \ref{sec4} we deal with linear and quasilinear Venttsel problems, respectively, and we assume here that the dimension of the underlying domain $\Omega$ is at least $3$. The 2D case is more simple, and we briefly discuss it in Section~\ref{sec5}. Some remarks about possible generalizations of our results are also given there, together with several open problems and indications for further research.

\section{Auxiliary results}\label{sec2}

\subsection{Notation and conventions}
Throughout the paper we use the following notation:

$x=(x',x_n)=(x_1,\dots,x_{n-1},x_n)$ is a vector in
${\mathbb R}^n$ with Euclidean norm $|x|;$

$\mathbb{R}^n_+=\{x\in \mathbb{R}^n : x_n>0\};$

$\Omega$ is a bounded domain in $\R^n$ with compact closure $\overline{\Omega}$ and $(n-1)$-dimensional boundary $\partial\Omega;$

$\Gamma (\Omega)$ is the part of $\partial \Omega$ lying on the hyperplane $\{x_n=0\};$

$\n=\n (x)$ is the unit vector of the outward normal to $\partial \Omega$ at the point $x;$

$B_r(x^0)$ is the open ball in $\mathbb{R}^n$ with center
$x^0$ and radius $r;$  $B_r=B_r(0);$  $B_r'(x^0)=B_r(x^0) \cap \{x_n=0\};$  $B_r'=B_r'(0);$

The  indices $i$ and $j$ run from $1$ to $n$ 
and we adopt the standard convention regarding summation with respect to repeated indices;

$D_i$ denotes the operator of (weak) differentiation with respect
to  $x_i;$

$Du=\left( D'u,D_nu\right) =\left( D_1u, \dots,D_{n-1}u,D_nu\right) $ is the gradient of $u;$

$d_i$ denotes the tangential differential operator on $\partial \Omega$, i.e.,
$$
d_i=D_i-\n_i \n_j D_j;
$$
$du=(d_iu)$ is the tangential gradient of $u$ on $\partial \Omega;$ in particular,  we have $du=(D'u,0)$ on $\Gamma (\Omega);$

$
\|\cdot\|_{p,\Omega}
$ denotes the norm in $L^p(\Omega);$

$W^{1}_p(\Omega)$ and 
$W^{2}_p(\Omega)$ are the Sobolev space with norms
$$
\|u\|_{W^1_p(\Omega)}=\|Du\|_{p,\Omega}+\|u\|_{p,\Omega}\quad \text{and}\quad 
\|u\|_{W^2_p(\Omega)}=\|D(Du)\|_{p,\Omega}+\|u\|_{p,\Omega}
$$
respectively; 
similarly, the symbols $W^1_p(\partial \Omega)$ and $W^2_p(\partial \Omega)$ stand for the So\-bo\-lev spaces of functions defined on $\partial\Omega$ and equipped with the corresponding norms
$$
\|u\|_{W^1_p(\partial \Omega)}=\|du\|_{p,\partial\Omega}+\|u\|_{p,\partial\Omega},\qquad 
\|u\|_{W^2_p(\partial \Omega)}=\|d(du)\|_{p,\partial\Omega}+\|u\|_{p,\partial\Omega};
$$

We also define $V_{p,q}(\Omega)$ as the subspace of $W^2_p(\Omega)$ consisting of all functions $u$  that have
 traces in $W^2_q(\partial \Omega)$, with the norm 
$$
\|u\|_{V_{p,q}(\Omega)}=\|u\|_{W^2_p (\Omega)}+\|u\|_{W^2_q (\partial\Omega)}.
$$

We denote by
${\mathcal C}(\overline\Omega)$ and ${\mathcal C}^1(\overline\Omega)$ the spaces of continuous and continuously differentiable functions, respectively. In a similar way, we introduce the spaces ${\mathcal C}(\partial\Omega)$ and ${\mathcal C}^1(\partial\Omega).$

${\mathcal C}^{0,\lambda}(\overline{\Omega})$ is the space of H{\"o}lder continuous functions with exponent $\lambda \in (0,1],$ and with norm 
$$
\|u\|_{{\mathcal C}^{0,\lambda}(\overline\Omega)}=\sup_{\Omega}|u|+\sup_{x,y\in \Omega}
\frac{|u(x)-u(y)|}{|x-y|^{\lambda}}.
$$
${\mathcal C}^{1,\lambda}(\overline{\Omega})$ is the space of functions with first-order  derivatives belonging to ${\mathcal C}^{0,\lambda}(\overline{\Omega})$.

For a functional space ${\mathcal X}$, we  understand the notation $\partial\Omega \in {\mathcal X}$ as follows. There is a positive $R_0$ such that for every $x^0 \in \partial\Omega$ the set $\partial\Omega \cap B_{R_0}(x^0)$ is (in a suitable Cartesian coordinate system) a graph $y_n=f(y')$ of a function $f\in {\mathcal X}$.
When saying that a given constant depends on 
``\textit{the properties of $\partial\Omega$}'' we simply mean dependence on the $\mathcal X$-norms of the 
diffeomorphisms that flatten 
locally $\partial\Omega,$  and on the area of $\partial\Omega.$

By $p^*$ we denote the Sobolev conjugate of the exponent $p$, that is,
$$
p^*:=
\begin{cases} 
\dfrac{np}{n-p} & \text{if}\ p<n,\\
\infty & \text{if}\ p\ge n.
\end{cases}
$$

We use the letters $C$ and $N$ (with or without indices) to denote various constants. To indicate that $C$ depends on some parameters, we list these in parentheses: $C(\dots)$. 
Finally, we set $\frac{0}{0}=0,$ if such an uncertainty occurs.

\subsection{\textit{VMO} functions}

We will deal here with differential operators with discontinuous principal coefficients belonging to the Sarason class of functions with mean oscillation that vanishes over shrinking balls.   In \cite{JN61} John and Nirenberg introduced the space of functions with {\it bounded mean oscillation (BMO).}  Their paper has been followed by various works exhibiting the importance of the \textit{BMO} functions in the harmonic analysis (see \cite{Sar75} and the references therein). Later, Sarason \cite{Sar75}  attracted the attention to a natural subspace of  \textit{BMO},  called \textit{VMO,} consisting of the functions with  {\it  vanishing mean oscillation}.  Let us give a precise definition of these spaces.
\begin{definition}(\cite{JN61,Sar75})\label{defVMO}
A locally  integrable function $f$ defined on $\R^n$ 
lies in \textit{BMO} if its integral oscillation  is bounded, that is, if
$$
\|f\|_*:=\sup_{B} \, \Xint-_{B} \big|f(x)-f_B\big|\,dx<\infty,
$$
where $B$ varies in the class of \textit{all} balls in $\mathbb{R}^n$ and $f_B$ stands for the integral average $|B|^{-1}\int_B f(x)\,dx.$ Modulo constant functions, 
\textit{BMO} becomes a Banach space under the norm 
 $\|f\|_*.$  

For a function $f\in \textit{BMO},$ define
\begin{equation}\label{eq_gamma}
\omega_f(r)=\sup_{\rho\le r} \, \Xint-_{B_\rho} \left|f(x)-f_{B_\rho}\right|\,dx,
\end{equation}
where $B_\rho$ varies now in the class of \textit{all} balls of radius $\rho.$ Then $f\in \textit{VMO}$ if 
$$
\lim_{r\to 0} \omega_f(r) =0
$$
and we refer to $\omega_f(r)$ as \textit{VMO}-modulus of  $f$.

For a bounded domain $\Omega\subset \R^n,$  the spaces $\textit{BMO}\,(\Omega)$  and $\textit{VMO}\,(\Omega)$   are defined in the same manner, replacing $B$ and $B_\rho$ above by the respective intersections with $\Omega.$ 
Similarly, if $\partial\Omega$ is smooth, the spaces $\textit{BMO}\,(\partial\Omega)$  and $\textit{VMO}\,(\partial\Omega)$ are defined in a natural way  by surface integral oscillations over $B\cap\partial\Omega$ and 
 $B_\rho\cap \partial\Omega$ with balls centered at points of $\partial\Omega.$
\end{definition}

Having a function $f\in \textit{VMO}\,(\Omega)$ given on a Lipschitz  domain,  it is possible to extend it to the whole $\R^n$ by preserving the \textit{VMO}-modulus as follows by results of  Acquistapace \cite{A92}.

It is worth to mention some examples that illustrate the embeddings between  \textit{VMO/BMO}  and  some classical function spaces.
Note that \textit{BMO} functions are not necessarily bounded but  the space of the bounded uniformly continuous  functions  belongs to \textit{VMO} and we can take as  \textit{VMO}-modulus the corresponding modulus of continuity. Further on, $W^1_n(\R^n)$ is a proper  subset of \textit{VMO} as it follows by the Poincar\'e inequality. However, if
 $f\in W^1_1(\R^n)$  with a gradient belonging to the Marcinkiewicz space $L^n_{\mathrm{weak}},$ then $f$ belongs to \textit{BMO} but not necessarily  to \textit{VMO}. This can be seen also by the following examples (cf. \cite{JN61}, \cite[Section~2.1]{MPS00}).          

Let $f(x)=\log|x|$ and set $f_\alpha(x)= |\log|x||^\alpha$ with $\alpha>0$.  Then
\begin{itemize}
\item{}    $f\in \textit{BMO}\setminus \textit{VMO}$, $\sin (f)\in \textit{BMO}\cap L^\infty;$
\item{}    $f_\alpha\in \textit{VMO}$  for each $\alpha\in(0,1)$, but $f_\alpha\in W^1_n(\R^n)$ \textit{only} if $\alpha\in (0, 1-\frac1n).$  
\end{itemize}

In \cite{Sar75}  Sarason gave alternative descriptions of \textit{VMO} which explain the wide  application  of this function space not only in the harmonic analysis but also  in the theory of PDEs,  stochastic theory,    etc. 
\begin{prop}\label{thSarason}
 {\em (\cite{Sar75})} For $f\in \textit{BMO}$, the following conditions are equivalent:
\begin{enumerate}
\item $f\in \textit{VMO};$
\item $f$ is in the \textit{BMO}-closure of bounded uniformly continuous functions;
\item $\lim_{y\to 0}\| f(\cdot-y) -f(\cdot)\|_*=0$.
\end{enumerate}
\end{prop}

Let us note that the last circumstance in Proposition~\ref{thSarason} guarantees the good behavior of the mollifiers of \textit{VMO} functions. Moreover, we are able  to approximate the \textit{VMO} functions with $\C^\infty_0$ functions.

\subsection{An extension result}

The next statement is a modification of Theorem~6.1 in \cite{AN95} and allows to extend Sobolev functions defined on $\partial\Omega$ to Sobolev functions in the whole $\Omega.$

\begin{thm} \label{extension-theorem}
Let the exponents $q$ and $p$ be chosen such that
$$
1\le p \le \frac{nq}{n-1} \quad \text{and} \quad q >1,
$$
and let $\partial\Omega \in \C^{1,1}$.

Then there exists an extension operator
$$
\Pi\colon\ W^2_q (\partial\Omega) \rightarrow W^2_p(\Omega)
$$
such that
\begin{equation} \label{6.1-AN95}
\|\Pi u\|_{W^2_p (\Omega)} \le N_0 \|u\|_{W^2_q(\partial\Omega)},
\end{equation}
where $N_0$ depends only on $p$, $q$ and the properties of $\partial\Omega$.
\end{thm}
\begin{proof} 
It is easy to see that it suffices to prove the theorem when $p=\frac{nq}{n-1}$.

\textit{Step 1.} We start with a  procedure constructing an extension operator from a flat boundary surface to a boundary strip that acts continuously from the space $\stackrel{\circ}{W}\vphantom{W}\!^2_q(B_R')$ into the space $W^2_p(B_R' \times (0,R))$.
Here $\stackrel{\circ}{W}\vphantom{W}\!^2_q(B_R')$ stands for the closure of $\mathcal{C}^\infty_0(B_R')$ with respect  to the norm in $W^2_q,$ and we assume that the function $u$ is extended as zero to the whole $\R^{n-1}.$

Now successive application of some statements from \cite{Tri78} yields the following\\[6pt]
\begin{tabular}{rlllll}
& \cite[formula (2.8.1/18)]{Tri78}: &\ & $W^2_q(\R^{n-1}) \rightarrow \mathbf{B}^{2-1/p}_{p,p}(\R^{n-1})$ & \ & (embedding), \\[2pt]
& \cite[Theorem 2.9.3 (a)]{Tri78}: &\ & $\mathbf{B}^{2-1/p}_{p,p}(\R^{n-1}) \rightarrow W^2_p (\R^n)$ & \ &
(extension). \\[-8pt]
& & & & &
\end{tabular}

\noindent
(Here the notation $\mathbf{B}$ of the Besov spaces corresponds to the book \cite{Tri78}).

Note that multiplying by a suitable cut-off function, 
the following properties can be ensured
\begin{enumerate}
\item the extended function is equal to $0$ for $|x_n|>R/2$;
\item if the initial function is equal to $0$ for $|x'|>R/2$, then the extended one is equal to $0$ for $|x'|>3R/4$.
\end{enumerate}
Moreover, the norm of the extension operator is bounded in terms of ${n}$, $q$, $p$ and $R$. 

\medskip

\textit{Step 2.} The condition $\partial\Omega \in \C^{1,1}$ implies that for each point $x^0 \in \partial\Omega$ there exists a Cartesian coordinate system with origin at $x^0$ satisfying the following conditions
\begin{enumerate}
\item the surface $\partial\Omega$ is tangent to the hyperplane $\{x_n=0\}$ at the point $x^0$;
\item the intersection of $\partial\Omega$ with the neighborhood $U_R=\big\{(x',x_n)\colon\ |x'|<R, |x_n|<R\big\}$ can be given by an equation $x_n=\omega (x')$ with $\omega \in W^2_\infty(B_R')$. 
\end{enumerate}
Moreover, the radius $R$ of the above neighborhood can be chosen one and the same for all points $x^0 \in \partial\Omega$. The change of variables $y'=x'$, $y_n=x_n-\omega (x')$ then maps $\partial\Omega \cap U_R$ into the ball of radius $R$ lying on the hyperplane $\{y_n=0\}$.

This change of the variables induces the ``transplantation'' operator $u(x) \rightarrow u(y)$ acting continuously from $W^2_q (\partial\Omega \cap U_R)$ to $W^2_q (B_R')$.
\medskip

\textit{Step 3.} Using the results of Steps 1 and 2 above, one can construct local extension operators that map $W^2_q (\partial\Omega)$-functions with sufficiently small $x$-support  into $W^2_p(\R^n)$-functions vanishing for $|x_n|>R/2$, $|x'|>3R/4$. Finally, the desired operator $\Pi$
can be glued  from the local operators via appropriate partition of unity.
\end{proof}

\section{The linear Venttsel problem}\label{sec3}

In the sequel we suppose that $n\ge3$.
Assume $\partial\Omega \in \C^{1,1}$ and
let the exponents $q$ and $p$ be chosen such that (see Fig.~\ref{fig1})
\begin{equation} \label{pq-cond}
1<p \le \frac{nq}{n-1} <p^*\quad \text{and} \quad 
q >1.
\end{equation}
\begin{figure}[h]
\centerline{\includegraphics[scale=.95]{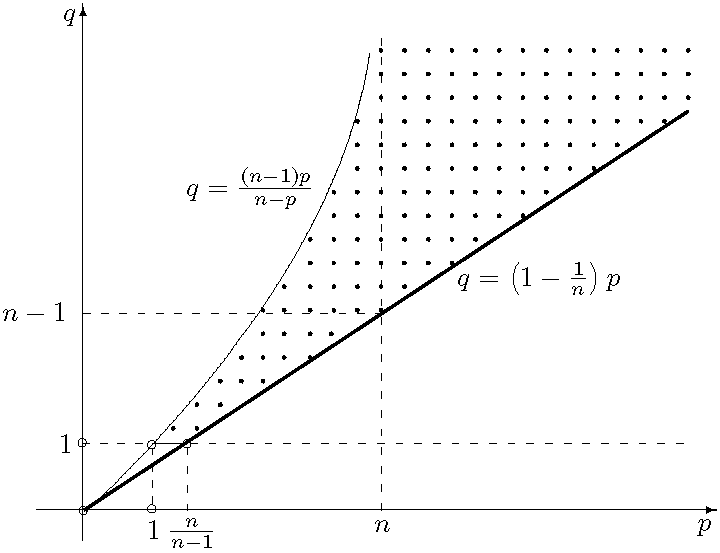}}
\vspace*{-.4cm}
\caption{The exponents $p$ and $q$}
\label{fig1}
\end{figure}

We introduce a linear elliptic operator $\CL$,
\begin{gather}
\label{4}
\CL \equiv -a^{ij}(x)D_iD_j+b^i(x)D_i+c(x),\\
\label{L1}
a^{ij}(x)=a^{ji}(x) \quad x\in \Omega, \qquad a^{ij} \in \textit{VMO}\,(\Omega), \tag{L1}\\
\label{L2}
\nu |\xi|^2 \le a^{ij}\xi_i\xi_j \le \nu^{-1} |\xi|^2 \quad \forall \xi\in \R^n, \quad \nu=\text{const}>0, \tag{L2}
\end{gather}
and a linear boundary operator $\CB$,
\begin{gather}
\label{5}
\CB \equiv -\alpha^{ij}(x)d_id_j+\beta^i(x)D_i+\gamma (x),\\
\label{B1}
\alpha^{ij}(x)=\alpha^{ji}(x) \quad x\in \partial\Omega, \qquad \alpha^{ij} \in \textit{VMO}\,(\partial\Omega), \tag{B1} \\
\label{B2}
\nu |\xi^*|^2 \le \alpha^{ij} \xi^*_i \xi^*_j \le \nu^{-1}|\xi^*|^2, \quad \forall\xi^*\in \R^n, \quad  \xi^* \perp \n (x). \tag{B2}
\end{gather}

Set $\mathbf{b}(x)=\big(b^1(x),\ldots,b^n(x)\big)$ and assume that the lower-order coefficients of the operator $\mathcal{L}$ satisfy the following integrability conditions
\begin{gather}
\label{L3}
\begin{aligned}
&|\mathbf{b}|\in L^{\max\{p,n\}}(\Omega), && \text{if} \quad p \neq n,\\
&|\mathbf{b}| \left( \log{(1+|\mathbf{b}|)}\right)^{1-1/n} \in L^n(\Omega), && \text{if} \quad p=n,
\end{aligned} 
\tag{L3}
\end{gather}
and
\begin{gather}
\label{L4}
\begin{aligned}
&c\in L^{\max\{p,n/2\}}(\Omega), &&  \text{if} \quad p \neq n/2,\\
&c\left(\log{( 1+|c|)}\right)^{1-1/n} \in L^{n/2}(\Omega), && \text{if} \quad p=n/2.
\end{aligned}
\tag{L4}
\end{gather}

The assumptions on the lower-order coefficients of the operator $\CB$ are as follows.   Most of the results obtained require the vector field $\boldsymbol{\beta} (x)=\big(\beta^1(x),\ldots,\beta^n(x)\big)$ to be an \textit{exterior field}  on $\partial\Omega$, that is,
\begin{equation} \label{B}
\beta_0(x):= \beta^i (x) \mathbf{n}_i(x) \ge 0, \qquad  x \in \partial\Omega.
\tag{B}
\end{equation}
We denote by $\boldsymbol{\beta}^*(x):= \boldsymbol{\beta} (x) - \beta_0(x)  \mathbf{n} (x)$ the tangential component of $\boldsymbol{\beta} (x)$ and assume 
\begin{gather}
\label{B3}
\begin{aligned}
&|\boldsymbol{\beta}^*|\in L^{\max\{q,n-1\}}(\partial\Omega), &&  \text{if} \quad q \neq n-1,\\
&|\boldsymbol{\beta}^*| \left( \log{(1+|\boldsymbol{\beta}^*|)}\right)^{1-1/(n-1)} \in L^{n-1}(\partial\Omega), && \text{if} \quad q=n-1,
\end{aligned}
\tag{B3}
\end{gather}
together with
\begin{gather}
\label{B4}
\begin{aligned}
&\gamma\in L^{\max\{q,(n-1)/2\}}(\partial\Omega),&& \text{if} \quad q \neq (n-1)/2,\\
&\gamma\left(\log{( 1+|\gamma|)}\right)^{1-1/(n-1)} \in L^{(n-1)/2}(\partial\Omega), && \text{if} \quad q=(n-1)/2.
\end{aligned}
\tag{B4}
\end{gather}

Finally, the normal component $\beta_0$ of the field  $\boldsymbol{\beta}$ is supposed to satisfy
\begin{gather}
\label{B5}
\begin{aligned}
&\beta_0\in L^{q}(\partial\Omega), && \text{if} \quad p >n,\\
&\beta_0\in L^{qp^*/(p^*-qn/(n-1))}(\partial\Omega), && \text{if} \quad p <n,\\
&\beta_0 \left( \log{(1+|\beta_0|)}\right)^{1-1/n} \in L^{q}(\partial\Omega), && \text{if} \quad p=n.
\end{aligned}
\tag{B5}
\end{gather}

Our first result provides an \textit{a~priori} estimate for any strong solution to the linear Venttsel problem in terms of the data of the problem.
\begin{thm} \label{apriori-estimate}
Let the exponents $p$ and $q$ be chosen in accordance with \eqref{pq-cond},   $\partial\Omega \in \C^{1,1}$
and assume that conditions \eqref{L1}--\eqref{L4} and 
\eqref{B1}--\eqref{B5} are verified.  

If a function $u\in V_{p,q}(\Omega)$  satisfies the equation
\begin{equation} \label{2.1-AN95}
\CL u(x)=f(x) \quad \text{a.e. in}\ \Omega
\end{equation}
and the boundary condition
\begin{equation} \label{2.2-AN95}
\CB u(x)=g (x) \quad \text{a.e. on}\ \partial\Omega
\end{equation}
with $f \in L^p(\Omega)$ and $g \in L^q(\partial\Omega),$
then
\begin{equation} \label{2.3-AN95}
\|u\|_{V_{p,q}(\Omega)} \le  C_1 \Big(\|f\|_{p,\Omega}+\|g\|_{q, \partial\Omega}+\|u\|_{p,\Omega}+
\|u\|_{q,\partial\Omega} \Big)
\end{equation}
with a constant $C_1$ depending  on $n$, $\nu$, $p$, $q$, $\mathrm{diam}\,\Omega$, the properties of $\partial\Omega$,
on the \textit{VMO}-moduli  of the coefficients $a^{ij}(x)$ and  $\alpha^{ij}(x)$, and on the moduli of continuity of the functions $|\mathbf{b}|$, $c$, $|\boldsymbol{\beta}^*|$, $\gamma$, and $\beta_0$ in the corresponding functional spaces defined by conditions \eqref{L3}--\eqref{L4} and 
\eqref{B3}--\eqref{B5}, respectively.\footnote{It is to be noted that the vector field $\boldsymbol{\beta}$ is {\bf not} required to be exterior to $\partial\Omega$ for the validity of Theorem~\ref{apriori-estimate}!}
\end{thm}

\begin{proof}
If $\beta_0 \equiv 0$ then $\beta^i(x)D_iu={\beta^*}^i(x)d_iu$ on $\partial\Omega$
and \eqref{2.2-AN95} can be considered as autonomous equation on $\partial\Omega:$ 
$$
-\alpha^{ij}(x)d_id_ju+{\beta^*}^i(x)d_iu+\gamma(x)u=g(x)\quad \text{a.e. on}\ \partial\Omega.
$$
Using the standard procedure of finite covering of $\partial\Omega$ by balls, local flattening of $\partial\Omega$ and employing there the 
 coercive estimates from \cite{CFL91,CFL93}, 
and putting finally these together with the aid of partition of unity, we get
$$
\|u\|_{W^2_q(\partial\Omega)} \le C\Big(\|u\|_{q,\partial\Omega}+\|g\|_{q,\partial\Omega}\Big).
$$
The function $u\in W^2_p(\Omega)$ solves the equation
$$
-a^{ij}(x)D_iD_ju+b^i(x)D_iu+c(x)u=f(x)\quad \text{a.e. in}\ \Omega
$$
and, according to Theorem~\ref{extension-theorem}, it assumes the boundary value $u|_{\partial\Omega}$ in the sense of $W^1_p.$ Then we apply once again the global $L^p$-theory of \cite{CFL93} in order to conclude
$$
\|u\|_{W^2_p(\Omega)} \le C\Big(\|u\|_{p,\Omega}+\|f\|_{p,\Omega}+\|u\|_{W^2_q(\partial\Omega)}\Big)
$$
that gives the claim  \eqref{2.3-AN95}. Actually, it is to be noted that the estimates in \cite{CFL93} regard second-order operators without lower-order terms, but appropriate use of interpolation inequalities leads to the same result for general second-order operators.

In the general case $\beta_0 \not\equiv 0$, we apply the so-called \textit{Munchhausen trick} 
(\cite{R86}, \cite[Theorem 2.1]{CLNV19}, see also \cite[Sect.~2]{AN95}) and estimate the directional derivative 
$\partial_{\mathbf{n}} u$ in terms of itself.  The procedure consists of three steps.

\medskip

\textit{Step 1.} 
Making use of 
$$
\beta^i(x)D_iu
={\beta^*}^i(x)d_iu+\beta_0(x)\partial_{\mathbf{n}}u,
$$ 
we rewrite the boundary condition \eqref{2.2-AN95} in the form
\begin{align} 
\label{2.14-AN95}
-\alpha^{ij}(x)d_id_ju &\ = g (x)-\beta_0 (x) 
\partial_{\mathbf{n}} u-{\beta^*}^i(x)d_iu-\gamma (x) u \\
\nonumber
&\ =: g_1 (x) \quad \text{a.e. on}\ \partial\Omega.
\end{align}

Now we consider \eqref{2.14-AN95} as an elliptic equation on $\partial\Omega$. As above, the standard procedure of finite covering of $\partial\Omega$ by balls, local flattening of $\partial\Omega$ and employing there the coercive estimates from  \cite{CFL91,CFL93},  putting these together via a partition of unity, implies that any solution $u$ of \eqref{2.14-AN95} satisfies the bound
$$
\|u\|_{W^2_q (\partial\Omega)} \le N_1 \Big(\|g_1\|_{q, \partial\Omega}+\|u\|_{q,\partial\Omega}\Big).
$$ 
Here $N_1$ is determined by $n$, $q$, $\nu$, the properties of $\partial\Omega$ and by the \textit{VMO}-moduli of the coefficients $\alpha^{ij}$.

Employing \eqref{2.14-AN95} in the last inequality, it takes on the form 
\begin{align} 
\label{2.15-AN95}
\|u\|_{W^2_q (\partial\Omega)} \le N_1 \Big(\|g\|_{q, \partial\Omega} +\|u\|_{q, \partial\Omega}&+ \left\|\beta_0 
\partial_{\mathbf{n}} u\right\|_{q,\partial\Omega}\\
\nonumber
&+\|{\beta^*}^i d_iu\|_{q, \partial\Omega}+\|\gamma u\|_{q,\partial\Omega}\Big).
\end{align} 
We proceed now to estimate  the last two terms in the right-hand of \eqref{2.15-AN95}. Take an arbitrary $\varepsilon >0$ and consider $\|{\beta^*}^i d_iu\|_{q, \partial\Omega}$.  The argument falls naturally into three possible cases. 

\medskip

\textit{Case 1a.} Let $q>n-1$. Since $W^2_q(\partial\Omega)$ is compactly embedded into $\C^1 (\partial\Omega)$, we have the estimate
\begin{align} 
\label{2.20-AN95}
\|{\beta^*}^id_iu\|_{q, \partial\Omega} \le&\ \|\boldsymbol{\beta}^*\|_{q, \partial\Omega} \|du\|_{\infty, \partial\Omega}\\
\nonumber
 \le&\ \varepsilon \|\boldsymbol{\beta}^*\|_{q, \partial\Omega}  \|u\|_{W^2_q(\partial\Omega)} + N_2(\varepsilon) \|\boldsymbol{\beta}^*\|_{q, \partial\Omega}  \|u\|_{q, \partial\Omega},
\end{align}
where $N_2(\varepsilon)$ depends also on $n$, $q$, $\mathrm{diam}\,\Omega$ and the properties of $\partial\Omega$. 

\medskip

\textit{Case 1b.} Let $q<n-1$. Now $|\boldsymbol{\beta}^*|\in L^{n-1}(\partial\Omega)$ and
we use the well-known idea (see, for example, \cite[Ch. III, \S 8, Remark 8.2]{LU68}) to decompose $|\boldsymbol{\beta}^*|$ into the sum
$$
|\boldsymbol{\beta}^*(x)|=\varphi_1(x)+\varphi_2(x),
$$
where  $\|\varphi_1\|_{n-1, \partial\Omega} \le \delta$ with a  small positive $\delta$ to be chosen later, and $\varphi_2 \in L^{\infty} (\partial\Omega)$. Notice that $\|\varphi_2\|_{\infty, \partial\Omega}$ is also determined by $\delta$ and by the moduli of continuity of $|\boldsymbol{\beta}^*|$ in $L^{n-1} (\partial\Omega)$. 
An application of the H{\"o}lder inequality yields 
\begin{equation} \label{2.20-2}
\|{\beta^*}^id_iu\|_{q, \partial\Omega} \le \|\varphi_1\|_{n-1, \partial\Omega} \|du\|_{q^*, \partial\Omega}+
\|\varphi_2\|_{\infty, \partial\Omega}\|du\|_{q,\partial\Omega},
\end{equation}
where $q^*=q(n-1)/(n-1-q)$. The first  term in \eqref{2.20-2} is estimated from above with the help of the Sobolev embedding on $\partial\Omega$, while the upper bound for the second term   follows from the compact embedding of $W^2_q(\partial\Omega)$ into $W^1_q(\partial\Omega)$. Thus, choosing $\delta$ small enough, we obtain
\begin{equation} \label{2.20-3}
\|{\beta^*}^id_i u\|_{q, \partial\Omega} \le \varepsilon \left( 1+\|\varphi_2\|_{\infty, \partial\Omega}\right) \|u\|_{W^2_q (\partial\Omega)}+N_3(\varepsilon)\|\varphi_2\|_{\infty, \partial\Omega}\|u\|_{q,\partial\Omega},
\end{equation}
where $N_3(\varepsilon)$ depends on the same parameters as $N_2(\varepsilon)$. 

\medskip

\textit{Case 1c.} If $q=n-1$ we argue in the same manner as in Case 1b. What is the difference now is that we use the Yudovich--Pohozhaev  embedding theorem  into the Orlicz space
$$
W^1_{n-1}(\partial\Omega) \hookrightarrow L_{\psi}(\partial\Omega)\quad \text{with}\quad \psi (t)=e^{|t|^{(n-1)/(n-2)}}-1
$$
(see, e.g., \cite[Sec.~10.5-10.6]{BIN75}).
 Therefore,
\begin{equation}\label{Orlicz}
|du|^{n-1}\in L_{\Psi}(\partial\Omega) \quad \text{with}\quad \Psi (t)=e^{|t|^{1/(n-2)}}-1,    
\end{equation}
and we observe that in the considered case the assumption \eqref{B3} ensures that $|\boldsymbol{\beta}^*|^{n-1}$ belongs to the Orlicz space $L_{\Psi^*} (\partial\Omega)$ dual to $L_{\Psi} (\partial\Omega)$, see \cite[Sec.~14]{KR58}. As a result we get again the estimate \eqref{2.20-3}, but now  
$\|\varphi_2\|_{\infty, \partial\Omega}$ is determined by the moduli of continuity of $|\boldsymbol{\beta}^*|$ in the Orlicz space related to \eqref{B3}.\footnote{Actually, the estimate \eqref{2.20-3} shows that $u\mapsto{\beta^*}^id_iu$ is a compact operator acting from $W^2_q (\partial\Omega)$ into $L^q(\partial\Omega),$ and further estimates of the lower-order terms can be interpreted in a similar way.}

\medskip

Summarizing, for sufficiently small $\varepsilon,$ 
we have 
\begin{equation} \label{2.20-final}
\|{\beta^*}^id_iu\|_{q, \partial\Omega} \le \frac{1}{4N_1} \|u\|_{W^2_q(\partial\Omega)} +\widehat{N}_1  \|u\|_{q, \partial\Omega}
\end{equation}
in the all three cases, 
where $N_1$ is the constant from \eqref{2.15-AN95}, while $\widehat{N}_1$ is  determined by $n$, $q$, $\mathrm{diam}\, \Omega$, the properties of $\partial\Omega$ and on the moduli of continuity of $|\boldsymbol{\beta}^*|$ in corresponding functional spaces defined by conditions \eqref{B3}.

Arguing in the same  way as in deriving of \eqref{2.20-final}, we conclude that
\begin{equation} \label{2.20-gamma}
\|\gamma u\|_{q, \partial\Omega} \le \frac{1}{4 N_1} \|u\|_{W^2_q(\partial\Omega)}+\widehat{N}_2  \|u\|_{q, \partial\Omega},
\end{equation}
where $\widehat{N}_2$ depends  on $n$, $q$, $\mathrm{diam}\, \Omega$, the properties of $\partial\Omega$ and on the moduli of continuity of $\gamma$ in corresponding functional spaces defined by conditions \eqref{B4}.

Substituting \eqref{2.20-final} and \eqref{2.20-gamma}  into the right-hand side of \eqref{2.15-AN95}, we obtain
\begin{equation} \label{2.15-AN95-short}
\|u\|_{W^2_q(\partial\Omega)} \le N_4 \left( \|g\|_{q,\partial\Omega}+\|u\|_{q,\partial\Omega}+
\left\|\beta_0 
\partial_{\mathbf{n}} u\right\|_{q,\partial\Omega} \right),
\end{equation}
where $N_4=2N_1 \left(1+\widehat{N}_1+\widehat{N}_2\right)$. 

\medskip

\textit{Step 2.}
Consider in $\Omega$ the function
\begin{equation} \label{2.16-AN95}
v(x)=u(x)-\widetilde{u}(x)
\end{equation}
with $\widetilde{u}=\Pi (u\big|_{\partial\Omega}),$ where $\Pi$ is the extension operator constructed in Theorem~\ref{extension-theorem}. It is evident that $v$ solves the boundary value problem
\begin{equation} \label{2.17-AN95}
\begin{cases}
-a^{ij}(x)D_iD_jv=f_1(x) &\quad \text{a.e. in}\ \Omega,\\
v=0 &\quad \hfill \text{a.e. on}\ \partial\Omega,
\end{cases}
\end{equation}
where $f_1(x):= f(x)-b^i(x)D_iu-c(x)u+a^{ij}(x)D_iD_j\widetilde{u}$.

It follows from Theorem~4.2 in \cite{CFL93}  that the solution of \eqref{2.17-AN95} satisfies the bound
$$
\|v\|_{W^2_p(\Omega)} \le N_5 \Big(\|f_1\|_{p,\Omega}+\|v\|_{p,\Omega}\Big),
$$
where $N_5$ depends only on $n$, $p$, $\nu$, on the properties of $\partial\Omega$ and on the \textit{VMO}-moduli of the coefficients $a^{ij}$. In view of \eqref{2.16-AN95}, \eqref{6.1-AN95} and the definition of $f_1$, one can transform the last inequality into
\begin{align} 
\label{2.18-AN95}
\|u\|_{W^2_p(\Omega)} \le N_6 \Big(\|f\|_{p,\Omega} +\|u\|_{p,\Omega}&+\|u\|_{W^2_q(\partial\Omega)} \\
\nonumber
&+\|b^iD_iu\|_{p,\Omega}+\|cu\|_{p,\Omega}\Big)
\end{align}
and $N_6$ depends on the same quantities as $N_5$. Repeating the arguments used in deriving \eqref{2.15-AN95-short}, we estimate the last two terms on the right-hand side of \eqref{2.18-AN95} and arrive finally at
\begin{equation} \label{2.18-AN95-short}
\|u\|_{W^2_p(\Omega)} \le N_7 \Big(\|f\|_{p,\Omega} +\|u\|_{p,\Omega}+\|u\|_{W^2_q(\partial\Omega)}\Big).
\end{equation}
Here $N_7$ is determined by the same parameters as $N_5$ and, in addition, it depends also on the moduli of continuity of $|\mathbf{b}|$ and $c$ in the corresponding functional spaces given by \eqref{L3} and \eqref{L4}, respectively

Combining  \eqref{2.15-AN95-short} with \eqref{2.18-AN95-short}, we get
\begin{align} \label{2.19-AN95}
\|u\|_{W^2_p(\Omega)} \le N_8 \Big(\|f\|_{p,\Omega} +\|g\|_{q,\partial\Omega}&+\|u\|_{p,\Omega}\\
\nonumber
&+\|u\|_{q,\partial\Omega}+
\left\|\beta_0 \partial_{\mathbf{n}} u\right\|_{q,\partial\Omega}\Big),
\end{align}
where $N_8=N_7 \left(1+N_4\right)$. 

\medskip

\textit{Step 3.}  We are in a position now to estimate the term $\left\|\beta_0 \partial_{\mathbf{n}} u\right\|_{q,\partial\Omega}$. We argue along the same lines as above when derived \eqref{2.15-AN95-short} and \eqref{2.18-AN95-short}. For $p>n$ we use the embedding $W^2_p(\Omega) \hookrightarrow \C^1 (\overline{\Omega})$, for $p<n$ and $p=n$ we use the trace embeddings of $W^1_p(\Omega)$ into $L^{p^*(n-1)/n}(\partial\Omega)$ and into the Orlicz space, respectively (see \cite[Sec.~10.5-10.6]{BIN75}). Thus, in the all three cases we have
\begin{equation} \label{beta_0}
\left\|\beta_0 \partial_{\mathbf{n}} u\right\|_{q,\partial\Omega} \le \frac{1}{2N_8} \|u\|_{W^2_p(\Omega)} +\widehat{N}_3  \|u\|_{p, \Omega},
\end{equation}
where $N_8$ is the constant from \eqref{2.19-AN95}, and $\widehat{N}_3$ is  determined by $n$, $p$, $\mathrm{diam}\, \Omega$, the properties of $\partial\Omega$ and on the moduli of continuity of $|\beta_0|$ in corresponding functional spaces defined by conditions \eqref{B5}. \medskip

Substituting  \eqref{2.19-AN95} into \eqref{beta_0}, we finalize the Munchhausen trick and arrive at
$$
\left\|\beta_0 \partial_{\mathbf{n}} u\right\|_{q,\partial\Omega} \le (1+2\widehat{N}_3) \Big(\|f\|_{p,\Omega} +\|g\|_{q,\partial\Omega}+\|u\|_{p,\Omega}+\|u\|_{q,\partial\Omega}\Big).
$$

Finally, inserting the last inequality into the right-hand sides of \eqref{2.15-AN95-short} and \eqref{2.19-AN95}  gives the claim \eqref{2.3-AN95} and this completes the proof.
\end{proof}

For the sake of further application of Theorem~\ref{apriori-estimate} to the study of quasilinear Venttsel problems, we need its variant concerning sequences of differential operators.
\begin{remark}\label{rem1}
\em
Let the sequence of operators
$$
\CL_k \equiv -a^{ij}_k(x)D_iD_j+b^i_k(x)D_i+c_k(x),\qquad
\CB_k \equiv -\alpha^{ij}_k(x)d_id_j+\beta^i_k(x)D_i+\gamma_k (x),
$$
satisfy the assumptions \eqref{L1}--\eqref{L2}, \eqref{B1}--\eqref{B2}. Suppose that the principal coefficients $a^{ij}_k$ and $\alpha^{ij}_k$ are \textit{VMO} functions uniformly in $k$, that is,
$$
\lim_{r\to 0} \sup_k\: \omega_{a^{ij}_k}(r)=\lim_{r\to 0} \sup_k\: \omega_{\alpha^{ij}_k}(r)=0.
$$
Define further $\beta_{0,k}(x):=\beta^i_k(x)\mathbf{n}_i(x)$ and assume that
$$
|\mathbf{b}_k|\le|\mathbf{b}|,\qquad
|c_k|\le|c|,\qquad
|\boldsymbol{\beta}^*_k|\le |\boldsymbol{\beta}^*|,\qquad
|\gamma_k|\le|\gamma|,\qquad
|\beta_{0,k}|\le|\beta_0|,
$$
where the functions $|\mathbf{b}|,c,|\boldsymbol{\beta}^*|,\gamma$ and $\beta_0$ satisfy the assumptions \eqref{L3}, \eqref{L4}, \eqref{B3}, \eqref{B4} and \eqref{B5}, respectively. Then the solutions of the boundary value problems 
$$
\CL_k u_k(x)=f_k(x) \quad \text{a.e. in}\ \Omega, \qquad
\CB_k u_k(x)=g_k (x) \quad \text{a.e. on}\ \partial\Omega
$$
satisfy the estimate
$$
\|u_k\|_{V_{p,q}(\Omega)} \le C_1 \Big(\|f_k\|_{p,\Omega}+\|g_k\|_{q, \partial\Omega}+\|u_k\|_{p,\Omega}+
\|u_k\|_{q,\partial\Omega} \Big)
$$
with a constant $C_1$ which is \textit{independent} of $k$.
\end{remark}

In the case when uniqueness theorem holds for the linear Venttsel problem \eqref{2.1-AN95}--\eqref{2.2-AN95},
the lower-order terms $\|u\|_{p,\Omega}$ and $\|u\|_{q,\partial\Omega}$ can be dropped from the right-hand side of \eqref{2.3-AN95}.
\begin{cor}
\label{corollary}
Let the domain $\Omega$ and the operators $\CL$ and $\CB$ in \eqref{2.1-AN95}--\eqref{2.2-AN95} satisfy all the assumptions of Theorem~$\ref{apriori-estimate}$. Assume also that the homogeneous problem 
\begin{equation} \label{homogen}
\CL u=0 \quad \text{a.e. in}\ \Omega, \qquad \CB u=0 \quad \text{a.e. on}\ \partial\Omega
\end{equation}
admits {\em only} the trivial solution in $V_{p,q}(\Omega)$. 

Then a solution $u\in V_{p,q}(\Omega)$ of \eqref{2.1-AN95}--\eqref{2.2-AN95} satisfies the estimate
\begin{equation} 
\label{2.3-AN95'}
\|u\|_{V_{p,q}(\Omega)} \le \widetilde C_1 \Big(\|f\|_{p,\Omega}+\|g\|_{q, \partial\Omega} \Big)
\end{equation}
with a constant $\widetilde C_1$ independent of $u$\footnote{In Lemma~\ref{dropped-est-h} below we give a more general result which regards sequences of differential operators.}.
\end{cor}
\begin{proof}
The proof is quite standard (see \cite[Lemma~9.17]{GT01} and the proof of Lemma~5.1 \cite{LM68} for a more general result). We argue by contradiction. If \eqref{2.3-AN95'} is false, then there is a sequence $u_k$ such that
$$
f_k:=\CL u_k\to0\ \ \textrm{in} \ \ L^p(\Omega), \quad g_k:=\CB u_k\to0\ \ \textrm{in} \ \ L^q(\partial\Omega),\quad \|u_k\|_{V_{p,q}(\Omega)}\equiv 1.
$$
Without loss of generality we may assume that $u_k\to u$ in $L^p(\Omega)$ and $u_k\to v$ in $L^q(\partial\Omega)$. The estimate \eqref{2.3-AN95} applied to pairwise differences $u_k-u_m$ yields
$$
\|u_k-u_m\|_{V_{p,q}(\Omega)}\to 0 \quad \textrm{as} \quad k,m\to\infty,
$$
and therefore $u_k\to u$ in $V_{p,q}(\Omega)$, $v=u|_{\partial\Omega}$. This, in turn, implies $\CL u=0$ and $\CB u=0,$ whence $u=0$ by the uniqueness, which is a contradiction.
\end{proof}

Standard arguments, based on the parameter continuation and the coercive estimate \eqref{2.3-AN95}, lead to the following existence theorem.
\begin{thm}\label{existence}
Under the hypotheses of Corollary $\ref{corollary},$ 
the non-homogeneous problem \eqref{2.1-AN95}--\eqref{2.2-AN95} admits a unique solution $u\in V_{p,q}(\Omega)$ for all $f \in L^p(\Omega)$ and $g \in L^q(\partial \Omega)$.
\end{thm}

Using this result we can prove that the couple $(\mathcal{L},\mathcal{B})$ supports 
the \textit{classical elliptic regularization property:}  if the data of
\eqref{2.1-AN95}--\eqref{2.2-AN95} allow the solution  to have better integrability, then the solution does gain it indeed.

\begin{thm}\label{regular}
Let the domain $\Omega$ and the operators $\CL$ and $\CB$ satisfy all the  assumptions of Theorem~$\ref{apriori-estimate}$. Suppose that the exponents $\tilde p<p$, $\tilde q<q$ verify  \eqref{pq-cond}. 

If $u\in V_{\tilde p, \tilde q}(\Omega)$ is a solution of \eqref{2.1-AN95}--\eqref{2.2-AN95} with $f \in L^p(\Omega)$ and $g \in L^q(\partial\Omega)$ then $u\in V_{p,q}(\Omega)$.
\end{thm} 

\begin{proof}
In case the homogeneous problem \eqref{homogen}
admits only the trivial solution in $V_{\tilde p,\tilde q}(\Omega),$ then uniqueness in $V_{\tilde p,\tilde q}(\Omega)$ implies uniqueness in $V_{\hat p,\hat q}(\Omega)$ for all $\hat p\in[\tilde p,p]$ and $\hat q\in[\tilde q,q]$  that satisfy \eqref{pq-cond}. Therefore,  Theorem~\ref{existence} gives the claim $u\in V_{p,q}(\Omega)$.

Otherwise, if the kernel of the couple $(\mathcal{L},\mathcal{B})$
in $V_{\tilde p,\tilde q}(\Omega)$ is non-trivial, 
then we hope to get at least the Fredholm property in $V_{\tilde p,\tilde q}(\Omega)$. Unfortunately, the hypotheses of Theorem \ref{apriori-estimate} ensure it in $V_{p,q}(\Omega),$ but generally not in $V_{\tilde p,\tilde q}(\Omega)$. The reason is that the requirement \eqref{B5}, in contrast to the other assumptions on the lower-order terms in \eqref{2.1-AN95}--\eqref{2.2-AN95}, becomes stronger when $p$ decreases.
To bypass that obstacle, we transfer the ``bad'' term from the left-hand side into the right one and rewrite the problem \eqref{2.1-AN95}--\eqref{2.2-AN95} as follows
\begin{align*}
& \widetilde\CL u:=(\CL+\lambda) u=f+\lambda u=:\tilde f & \quad \text{a.e. in}\ \Omega;\\
& \widetilde\CB u:=-\alpha^{ij}d_id_ju+{\beta^*}^id_iu+\gamma u=g-\beta_0\partial_{\mathbf{n}}u=:\tilde g & \quad \text{a.e. on}\ \partial\Omega.
\end{align*} 
The operator $\widetilde{\CB}$ has zero normal derivative component, and our hypotheses  ensure the Fredholm property in $V_{\tilde p,\tilde q}(\Omega)$.  
Choosing $\lambda$ in a way to avoid the discrete spectrum of the couple $(\widetilde{\CL},\widetilde{\CB})$, the above problem results uniquely solvable in $V_{\tilde p,\tilde q}(\Omega)$ for any $\tilde f \in L^{\tilde p}(\Omega)$ and $\tilde g \in L^{\tilde q}(\partial\Omega)$. Again, uniqueness in $V_{\tilde p,\tilde q}(\Omega)$ implies uniqueness in $V_{\hat p,\hat q}(\Omega)$ for all $\hat p\in[\tilde p,p]$ and $\hat q\in[\tilde q,q]$ that satisfy \eqref{pq-cond}.
\begin{figure}[h]
\centerline{\includegraphics[scale=.95]{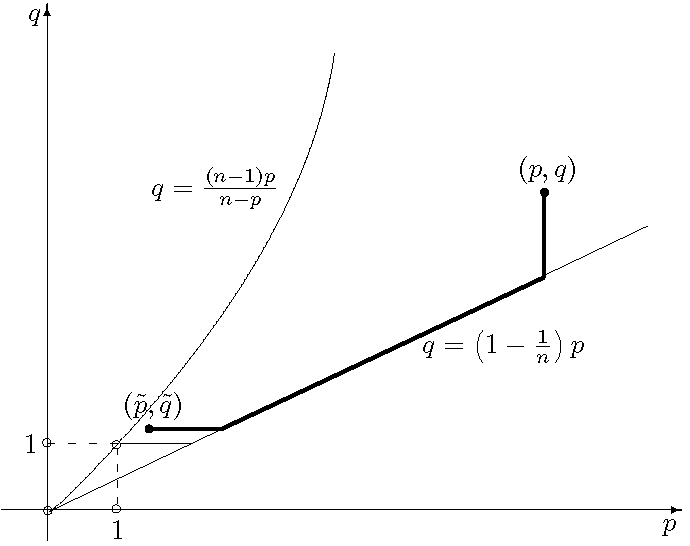}}
\vspace*{-.4cm}
\caption{The passage from $(\tilde p,\tilde q)$ to $(p,q)$}
\label{fig2}
\end{figure}

In general, the assumption \eqref{B5} does not guarantee $\beta_0\partial_{\mathbf{n}}u\in L^{\tilde q}(\partial\Omega)$ for arbitrary $u\in V_{\tilde p,\tilde q}(\Omega)$. Nevertheless, since all terms of $\widetilde{\CB}u$ belong to $L^{\tilde q}(\partial\Omega),$ we conclude $\tilde g\in L^{\tilde q}(\partial\Omega)$.  Moreover, $\tilde f\in L^{\hat p}(\Omega)$ with $\hat p=\min\{p,(\tilde p^*)^*\},$ and since $\tilde p$ and $\tilde q$ satisfy \eqref{pq-cond}, we have $(\tilde p^*)^*>\frac{n\tilde q}{n-1}.$ At that point Theorem~\ref{existence} yields $u\in V_{{\mathfrak p},\tilde q}(\Omega)$ with ${\mathfrak p}=\min\{p, \frac{n\tilde q}{n-1}\}$
(the horizontal tract of the thick line on Fig.~\ref{fig2}).

Now we move back the normal derivative term in the left-hand side and rewrite \eqref{2.1-AN95}--\eqref{2.2-AN95} as 
\begin{equation*}
\widetilde\CL u=\tilde f \quad \text{a.e. in}\ \Omega;\qquad
\CB u=g  \quad \text{a.e. on}\ \partial\Omega.
\end{equation*}
For this problem, the assumptions  ensure the Fredholm property in $V_{\hat p,\hat q}(\Omega)$ for any $\hat p\in[{\mathfrak p}, p],$ $\hat q=(1-\frac{1}{n})\hat p,$ and in $V_{p,\hat q}(\Omega)$ for any $\hat q\in\left[(1-\frac{1}{n})p,q\right].$ So, we can choose again $\lambda$ in a way that this problem is uniquely solvable in $V_{{\mathfrak p},\tilde q}(\Omega)$ for any $\tilde f \in L^{\mathfrak p}(\Omega)$ and $\tilde g \in L^{\tilde q}(\partial\Omega)$. Repeating the previous arguments, we get successively $u\in V_{p,(1-\frac{1}{n})p}(\Omega)$ and $u\in V_{p,q}(\Omega)$ (the oblique and vertical tracts on Fig.~\ref{fig2}).
\end{proof}

Theorem~\ref{existence} shows that uniqueness is a sufficient condition guaranteeing existence of strong solutions to the linear Venttsel problem \eqref{2.1-AN95}--\eqref{2.2-AN95}.
There are various types of additional requirements to impose on the coefficients of  $\CL$ and $\CB$ that ensure the validity of  global maximum principle, and hence triviality of the kernel of \eqref{homogen}. For instance, the following statement holds true.
\begin{thm} \label{global-max-principle}
Let $\partial\Omega \in W^2_n$ and let the operators $\CL$ and $\CB$ be defined by the formulas \eqref{4} and \eqref{5}, respectively. Suppose that $\{a^{ij}\}$ and $\{\alpha^{ij}\}$ are symmetric matrices and that hypotheses \eqref{L2}, \eqref{B2} and \eqref{B} are fulfilled.

Assume also that 
$$
|{\mathbf{b}}| \in L^n (\Omega), \qquad |{\boldsymbol{\beta}}^*| \in L^{n-1}(\partial\Omega)
$$
and
\begin{equation}\label{c-gamma}
c \ge 0 \quad {\text{a.e. in}} \ \Omega, \qquad
\gamma \ge \gamma_0 \quad {\text{a.e. on}} \ \partial\Omega, \quad
\gamma_0={\rm{const}}>0.
\end{equation}

If $u \in V_{n, n-1}(\Omega)$ satisfies 
$$
\CL u \le 0 \quad \text{a.e. in}\ \Omega, \qquad \CB u\le 0 \quad \text{a.e. on}\ \partial\Omega,
$$
then $u \le 0$ in $\overline\Omega$. In particular, 
the problem \eqref{homogen} admits only the trivial solution in the space $V_{n, n-1}(\Omega)$.
\end{thm}

\begin{proof}
We argue by contradiction. 
Note that $u \in V_{n, n-1}(\Omega)$ implies $u\in \mathcal{C}(\overline{\Omega})$ and let $M:=\max_{\overline\Omega}u>0$.
By the Aleksandrov--Bakel'man maximum principle (\cite{Al63}, see also Theorem~1.5 in the survey paper  \cite{N05}), the maximum of $u$ is achieved at a point $x^0 \in \partial\Omega$. We take a coordinate system centered at the point $x^0$ and flatten $\partial\Omega$ in a neighborhood of $x^0$, so that $\Omega\cap B_R\subset \mathbb R^n_+$ for some $R>0$. It is worth noting that all the assumptions of Theorem~\ref{global-max-principle} are invariant with respect to this coordinate transformation.

Further, we put $C=\nu^{-1}\sqrt{n-1}$ and  introduce the set
$$
\mathcal{O}_\rho =\big\{x\in{\mathbb R}^n\colon\quad |x'|<\rho ,\ 0<x_n<C^{-1}\rho \big\}
$$
and the function
$$
u_{\varepsilon}(x)=u(x)-M+\varepsilon\left(1-\frac {|x'|^2}{\rho ^2}+\frac {C^2x_n^2}{\rho ^2}
-\frac {2Cx_n}{\rho }\right),
$$
where $\rho<R$ and $\varepsilon <M$ are positive parameters to be chosen later.

It is evident that $u_{\varepsilon}|_{\partial \mathcal{O}_\rho  \setminus \Gamma(\mathcal{O}_\rho )}\le 0$ and $u_{\varepsilon}(0)=\varepsilon>0$. Applying the local Aleksandrov-type estimate for the Venttsel problem  \cite[Theorem~3]{AN95a} (see also \cite[Theorem~3.1]{N05}) to the function $u_{\varepsilon}$ in $\mathcal{O}_\rho $, we obtain 
\begin{align*}
\varepsilon \le&\  \widetilde C\big(n, \nu, \|{\mathbf b}\|_{n,\mathcal{O}_\rho },
\|{\boldsymbol{\beta}}^*\|_{n-1,\Gamma(\mathcal{O}_\rho )} \big)\times   \\
&\ \times\rho\left[\frac {2\varepsilon C}{\rho }
\|{\bf b}\|_{n,\mathcal{O}_\rho }
+\left\|\left(-\gamma_0(M-\varepsilon)+\frac {2\varepsilon (n-1)}{\nu \rho ^2}+\frac {2\varepsilon}{\rho }
|{\boldsymbol{\beta}}^*|\right)_+\right\|_{n-1,\Gamma(\mathcal{O}_\rho )}\right].
\end{align*}
Thus, for $\varepsilon<\frac M2\big(1+\frac {2(n-1)}{\gamma_0\rho ^2\nu}\big)^{-1}$,
 we have 
\begin{equation} \label{uzhas}
1\le 2\widetilde C \left[C
\|{\bf b}\|_{n,\mathcal{O}_\rho }
+\left\|\left(|{\boldsymbol{\beta}}^*|-\frac {\gamma_0M\rho }{4\varepsilon}
\right)_+\right\|_{n-1,\Gamma(\mathcal{O}_\rho )}\right].
\end{equation}

The first term in the square brackets in \eqref{uzhas} tends to zero as $\rho  \to 0$. Therefore, there exists a value of $\rho $ such that
$$
\frac 1{4\widetilde C}\le \left\|\left(
|{\boldsymbol{\beta}}^*|-\frac {\gamma_0M\rho }{4\varepsilon}
\right)_+\right\|_{n-1,\Gamma(\mathcal{O}_\rho )}.
$$
However, the right-hand side of the last inequality tends to zero as $\varepsilon \to 0$ and the contradiction obtained gives the claim.
\end{proof}

\begin{cor}\label{solv}
Let the exponents $p$ and $q$ satisfy \eqref{pq-cond},   $\partial\Omega \in \C^{1,1},$
and assume that conditions \eqref{L1}--\eqref{L2}, \eqref{B1}--\eqref{B2}, \eqref{B} and \eqref{c-gamma} are satisfied. 

Suppose also that
\begin{gather*}
\begin{aligned}
&c, |\mathbf{b}| \in L^p(\Omega), && \text{if} \quad p > n,\\
&c, |\mathbf{b}| \left( \log{(1+|\mathbf{b}|)}\right)^{1-1/n} \in L^n(\Omega), && \text{if} \quad p\le n;
\end{aligned} 
\end{gather*}
\begin{gather*}
\begin{aligned}
&\gamma, |\boldsymbol{\beta}^*|\in L^q(\partial\Omega), &&  \text{if} \quad q > n-1,\\
&\gamma, |\boldsymbol{\beta}^*| \left( \log{(1+|\boldsymbol{\beta}^*|)}\right)^{1-1/(n-1)} \in L^{n-1}(\partial\Omega), && \text{if} \quad q\le n-1;
\end{aligned}
\end{gather*}
\begin{gather*}
\begin{aligned}
&\beta_0\in L^{q}(\partial\Omega), && \text{if} \quad p >n,\\
&\beta_0\in L^{qp^*/(p^*-qn/(n-1))}(\partial\Omega), && \text{if} \quad p \le n,\ \ q>
\left(1-\frac{1}{n}\right)p,\\
&\beta_0 \left( \log{(1+|\beta_0|)}\right)^{1-1/n} \in L^{n-1}(\partial\Omega),
 && \text{if} \quad \frac n{n-1}<p \le n,\ \ q=\left(1-\frac{1}{n}\right)p.
\end{aligned}
\end{gather*}
Then the problem \eqref{2.1-AN95}--\eqref{2.2-AN95} is uniquely solvable in
$V_{p,q}(\Omega)$ for any $f \in L^p(\Omega)$ and $g \in L^q(\partial\Omega).$
\end{cor}

\begin{proof}
If $p\geq n$ and $q\geq n-1,$ then Theorem~\ref{global-max-principle} ensures triviality of the solution of \eqref{homogen} in $V_{n,n-1}(\Omega)$ and thus in $V_{p,q}(\Omega)$ as well,  and the claim follows from Theorem~\ref{existence}.

Otherwise, the integrability requirements on the lower-order coefficients of $\mathcal{L}$ and $\mathcal{B}$  guarantee that any solution $u\in V_{p,q}(\Omega)$ of \eqref{homogen} in fact belongs to $V_{n,n-1}(\Omega)$ through Theorem~\ref{regular}. Then the desired unique solvability  follows once again from \eqref{c-gamma} and Theorems~\ref{global-max-principle} and \ref{existence}.
\end{proof}

\section{The quasilinear Venttsel problem}
\label{sec4}

We aim now to the study of quasilinear elliptic equation
\begin{equation} \label{1.1-AN95}
-a^{ij}(x,u)D_iD_ju+a(x,u,Du)=0 \quad \text{a.e. in}\ \Omega
\end{equation}
coupled with the quasilinear Venttsel boundary condition
\begin{equation} \label{1.2-AN95}
-\alpha^{ij}(x,u)d_id_iu+\alpha (x,u,Du)=0 \quad \text{a.e. on}\ \partial\Omega
\end{equation}
over domains $\Omega$ with $\mathcal{C}^{1,1}$-smooth boundaries. As in the previous Section, we suppose that $n\geq 3.$

Notice that \eqref{1.2-AN95} is not an autonomous equation on $\partial\Omega$ because it involves not only tangential derivatives but also the normal component of the gradient $Du$.

We suppose that the functions $a^{ij}(x,z)$ and $a(x,z,p)$ are Carath\'{e}odory functions, i.e., these are measurable in $x \in \Omega$ for all $(z,p) \in \mathbb{R}\times\mathbb{R}^n$ and continuous with respect to $z$ and $p$ for almost all $x\in \Omega$. 
The equation \eqref{1.1-AN95} will be assumed to be uniformly elliptic, that is, for almost all $ x\in \Omega$ and for all $z\in \mathbb{R}$ we have
\begin{equation} 
\label{A1}
\begin{gathered} 
\nu |\xi|^2 \le a^{ij}(x,z)\xi_i\xi_j \le \nu^{-1} |\xi|^2 \quad \forall \xi\in \R^n, \ \nu=\text{const}>0, \quad a^{ij}=a^{ji}.
\tag{A1}
\end{gathered}
\end{equation}

Regarding the regularity conditions of the coefficients $a^{ij},$ we suppose that
\begin{equation} 
\label{A2}
\begin{gathered}
a^{ij}(\cdot,z)\in \textit{VMO}\   \textit{locally uniformly in}\ z, \textit{that is}, \\ 
\lim_{r\to 0} \sup_{z\in [-M,M]} \omega_{a^{ij}(\cdot,z)}(r)=0,
\tag{A2}
\end{gathered}
\end{equation}
where $\omega_{a^{ij}(\cdot,z)}$ is the \textit{VMO}-modulus of $a^{ij}(\cdot,z)$ defined by \eqref{eq_gamma} with $B_{\rho}$ replaced by $B_{\rho} (x) \cap \Omega$, $x\in \Omega$. Moreover, we need $a^{ij}$ to be   
locally uniformly continuous of with respect to $z,$ uniformly in $x:$
\begin{equation} 
\label{A3}
\begin{gathered}
|a^{ij}(x,z)-a^{ij}(x,\widehat{z})|\le \tau_M(|z-\widehat{z}|)\quad \textit{a.e. in}\ \Omega,\quad
\forall  z, \widehat{z}\in [-M,M]\\ 
\textit{with a non-decreasing function}\ 
\tau_M(t),\ \lim_{t \to 0}\tau_M (t)=0.
\tag{A3}
\end{gathered}
\end{equation}

The function $a(x,z,p)$ is assumed to grow quadratically  with respect to the gradient, i.e., for almost all $x\in \Omega$ and for all $(z,p)\in \mathbb{R}\times\mathbb{R}^n$
\begin{equation} 
\label{A4}
|a(x,z,p)|\le \eta(|z|)\Big(\mu |p|^2+b(x)|p|+\Phi(x)\Big)
\tag{A4}
\end{equation}
with a constant $\mu\ge0$ and a non-decreasing function $\eta\in\mathcal{C}(\mathbb R_+),$ and where 
\begin{equation} \label{A5}
 b \left(\log{(1+|b|)}\right)^{1-1/n} \in L^n(\Omega),\quad \Phi \in L^n(\Omega).
\tag{A5}
\end{equation}

Further on, we assume that the boundary condition 	\eqref{1.2-AN95} is a uniformly elliptic Venttsel condition in the sense that  
for almost all $x\in \partial\Omega$ and for all $(z,p) \in \mathbb{R}\times \mathbb{R}^n$ we have:
\begin{equation} \label{V}
\begin{gathered}
\textit{the function}\ \alpha(x,z,p)\ \textit{is weakly differentiable with respect to}\ p_i,\ \textit{and}\\
0 \le \alpha_{p_i}(x,z,p)\mathbf{n}_i(x) \le \eta(|z|)\beta_0(x),
\tag{V}
\end{gathered}
\end{equation}
with $\eta$ as above and
\begin{equation} \label{V0}
\beta_0 \left(\log{(1+|\beta_0|)}\right)^{1-1/n} \in L^{n-1}(\partial\Omega)
\tag{V0}
\end{equation}
and
\begin{equation} \label{V1}
\begin{gathered}
\nu |\xi^*|^2 \le \alpha^{ij}(x,z)\xi^*_i\xi^*_j \le \nu^{-1} |\xi^*|^2 \quad \forall \xi^*\in \R^n, \quad \xi^* \perp \mathbf{n}(x),\\ \nu=\text{const}>0, \quad \alpha^{ij}=\alpha^{ji},
\tag{V1}
\end{gathered}
\end{equation}
In addition, we impose regularity  conditions on the coefficients $\alpha^{ij}$ similar to these required for $a^{ij}$. Precisely,
\begin{equation} \label{V2}
\begin{gathered}
\alpha^{ij}(\cdot ,z)\in \textit{VMO}\   \textit{locally uniformly in}\ z, \ \textit{that is},\\
\lim_{r\to 0} \sup_{z\in [-M,M]} \omega_{\alpha^{ij}(\cdot,z)}(r)=0,
 \tag{V2}
\end{gathered}
\end{equation}
where $\omega_{\alpha^{ij}(\cdot ,z)}$ is the \textit{VMO}-modulus of $\alpha^{ij}(\cdot,z)$ defined by \eqref{eq_gamma} with $\partial\Omega \cap B_{\rho} (x)$, $x\in \partial\Omega$, in the place of $B_{\rho}$,
and
\begin{equation} \label{V3}
\begin{gathered}
|\alpha^{ij}(x,z)-\alpha^{ij}(x,\widehat{z})|\le {\tau}_M(|z-\widehat{z}|)\quad \textit{a.e. on}\ \partial\Omega,
\quad \forall  z, \widehat{z}\in [-M,M],\\ \textit{with}\ 
{\tau}_M(t)\
\textit{as in}\ \eqref{A3}.
\tag{V3}
\end{gathered}
\end{equation}

The quasilinear term $\alpha(x,z,p)$ of \eqref{1.2-AN95} is required to support quadratic growth with respect to the tangential gradient, i.e., for almost all $x\in \partial\Omega$ and for all $z\in \mathbb{R}$
\begin{equation} \label{V4}
|\alpha(x,z,p^*)|\le \eta(|z|)\Big(\mu |p^*|^2+\beta (x)|p^*|+\Theta (x)\Big)\quad \forall p^*\in\mathbb{R}^n,\quad p^*\perp \mathbf{n}(x)
\tag{V4}
\end{equation}
with $\mu$ and $\eta$ as in \eqref{A4} and where
\begin{equation} \label{V5}
\beta \left(\log{(1+|\beta|)}\right)^{1-1/(n-1)} \in L^{n-1}(\partial\Omega),\quad 
\Theta \in L^{n-1}(\partial\Omega).
\tag{V5}
\end{equation}

It is well known (see \cite[Lemma~2.1]{P95}) that conditions \eqref{A2}--\eqref{A3} and \eqref{V2}--\eqref{V3} provide for $u(x) \in \C(\overline{\Omega})$ the inclusions $a^{ij}(x, u(x)) \in \textit{VMO} \cap L^{\infty}(\Omega)$ and $\alpha^{ij}(x,u(x)) \in \textit{VMO} \cap L^{\infty}(\partial\Omega)$ with \textit{VMO}-moduli bounded in terms of the continuity modulus of $u(x)$ and $\sup_\Omega|u|$.\medskip

\smallskip

The strong solvability of the problem \eqref{1.1-AN95}--\eqref{1.2-AN95} will be proved in  $V_{n,n-1}(\Omega)$ by the aid of the Leray--Schauder fixed point theorem. To apply it, we have to derive \textit{a~priori} estimates in a suitable functional space for any solution to a family of quasilinear Venttsel problems. Following the classical approach of O.A.~Ladyzhenskaya and N.N.~Ural'tseva~\cite{LU86}, we obtain these estimates assuming  we already dispose of a bound for the supremum norm $\sup_\Omega|u|$ of a solution $u \in  V_{n,n-1}(\Omega)$.

\medskip

We recall, first of all, the \textit{a~priori} estimate for the H{\"o}lder norm of a solution.
\begin{prop}[Theorem 1$'$ in \cite{AN95a}] \label{Holder-est}
Let $\partial\Omega \in W^2_n$. Suppose that the function $u \in  V_{n,n-1}(\Omega)$
is a solution of \eqref{1.1-AN95}--\eqref{1.2-AN95}. 

Assume also that conditions \eqref{A1}, \eqref{A4}, \eqref{V}, \eqref{V1} and \eqref{V4} are satisfied with
$$
b,\Phi\in L^n(\Omega);\qquad \beta_0,\beta,\Theta\in L^{n-1}(\partial\Omega).
$$

Then there exists a constant $\lambda >0$ depending only on $n$, $\nu$ and the properties of $\partial\Omega$, such that 
$$
\|u\|_{\mathcal{C}^{0,\lambda} (\overline\Omega)}\le M_{\lambda},
$$
where $M_{\lambda}$ depends only on $n$, $\nu$, $\mu$,  the properties of $\partial\Omega$, $M_0=\sup_{\Omega}|u|$, ${\eta}(M_0)$, $\|\Phi\|_{n,\Omega}$, $\|\Theta\|_{n-1,\partial\Omega}$ and on the moduli of continuity of the functions $b$, $\beta_0$ and $\beta$ in the corresponding Lebesgue spaces. 
\end{prop}

The key \textit{a~priori} estimate of our approach is the gradient one.

\begin{thm}\label{Du-est}
Let $\partial\Omega \in \C^{1,1}$ and assume  that conditions \eqref{A1}--\eqref{A5}, \eqref{V}, \eqref{V0}--\eqref{V5} are satisfied.

Then any solution
 $u \in  V_{n,n-1}(\Omega)$ of the problem
\eqref{1.1-AN95}--\eqref{1.2-AN95} fulfills
the estimate
\begin{equation}
\label{grad-est}
\|Du\|_{2n,\Omega}+\|du\|_{2(n-1),\partial\Omega} \le M_1
\end{equation}
with a constant $M_1$ depending on:
\begin{itemize}
    \item $n$, $\nu$, $\mathrm{diam}\,\Omega$ and the properties of $\partial\Omega$;
    \item $M_0=\sup_{\Omega}|u|$ and ${\eta}(M_0)$;
    \item the norms $\|\Phi\|_{n,\Omega}$ and $\|\Theta\|_{n-1,\partial\Omega}$;
    \item the constant $\mu$ and the moduli of continuity of the functions $b$, $\beta_0$ and $\beta$ in the Orlicz spaces defined by conditions \eqref{A5}, \eqref{V0} and \eqref{V5}, respectively;
    \item the \textit{VMO}-moduli w.r.t. $x$ and on the moduli of continuity w.r.t. $z$  of the Carath\'eodory functions $a^{ij}(x,z)$ and $\alpha^{ij}(x,z)$, see conditions \eqref{A2}--\eqref{A3}, \eqref{V2}--\eqref{V3}.
\end{itemize}
\end{thm}

\begin{proof}
The \textit{a~priori} estimate \eqref{grad-est} will be derived with the aid of a homotopy technique which goes back to Amann and Crandall \cite{AC78} and has been used in \cite{P95} and \cite{FaP96b} in the study of the Dirichlet and oblique derivative problems for quasilinear elliptic operators with discontinuous coefficients.

For the sake of brevity, we define hereafter
\begin{align}
\tilde{a}^{ij}(x):=&\ {a}^{ij}(x,u(x)),
\label{tilde-aij}\\ 
 \tilde{a}(x):=&\ \dfrac{a(x,u(x),Du(x))}{\mu|Du(x)|^2+b(x)|Du(x)|+\Phi(x)}, \label{tilde-a}\\
\tilde{b}^i(x):=&\ 
b(x)\tilde{a}(x)\, \dfrac{D_iu(x)}{|Du(x)|}
\label{tilde-b}
\end{align}
(recall that we set $\frac{0}{0}=0$, if such an uncertainty occurs)
and note that the assumption \eqref{A1} implies immediately the uniform ellipticity condition \eqref{L2} for $\tilde{a}^{ij}$.

Further, we make use of the hypotheses \eqref{A2}, \eqref{A3} and employ \cite[Lemma~2.1]{P95} in order to get that the \textit{VMO}-moduli of $\tilde{a}^{ij}(x)$  and $\tilde{\alpha}^{ij}(x)$ are controlled in terms of $M_0=\sup_\Omega|u|$ and the modulus of continuity of $u$. Further, Proposition~\ref{Holder-est} provides estimates 
 the continuity modulus of $u$ in terms of $M_0$ and therefore
 $\tilde{a}^{ij}$ satisfy \eqref{L1}. 
Moreover, $|\tilde a|\le {\eta}(M_0)$ in view of \eqref{A4}
while $\tilde b^i$ verify \eqref{L3} as consequence of \eqref{A5}.

This way, the equation \eqref{1.1-AN95} can be rewritten as
\begin{equation}\label{VP1}
-\tilde{a}^{ij}(x)D_iD_ju+\tilde{b}^i(x)D_iu+\mu\tilde{a}(x)|Du|^2+\Phi(x) u
=\tilde{f}(x)
\quad \text{a.e. in}\ \Omega
\end{equation}
with
$$
\tilde{f}(x):= \Phi(x)\big(u(x)-\tilde{a}(x)\big)\in L^n(\Omega).
$$

Similarly, we define
\begin{align}
\label{tilde-alpij}
\tilde{\alpha}^{ij}(x):=&\ {\alpha}^{ij}(x,u(x)), \\ 
\label{tilde-alpha}
 \tilde{\alpha}(x):=&\ \dfrac{\alpha(x,u(x),du(x))}{\mu|du(x)|^2+\beta(x)|du(x)|+\Theta(x)},  \\
\label{t-beta-prime} 
\tilde{\beta}^{*i}(x):=&\ 
\beta(x)\tilde{\alpha}(x)\, \dfrac{d_iu(x)}{|du(x)|}, \\
\label{beta0u}
\tilde\beta_0(x):= &\ \int_0^1 \alpha_{p_i}\big(x,u(x),du(x)+t{\mathbf{n}}\partial_{\mathbf{n}}u(x)\big)\mathbf{n}_i(x)\,dt 
\end{align}
and note that  the assumptions \eqref{V1}--\eqref{V3} and \cite[Lemma~2.1]{P95} imply the conditions \eqref{B1}--\eqref{B2} for $\tilde{\alpha}^{ij}$. Further on, $|\tilde \alpha|\le {\eta}(M_0)$ by \eqref{V4}, $\tilde\beta^{*i}$ satisfy \eqref{B3} because of \eqref{V5}, while \eqref{V} implies
$0\le \tilde\beta_0(x)\le {\eta}(M_0)\beta_0(x)$
for a.a. $x\in\partial\Omega.$ 
Moreover,
$$
\alpha(x,u,Du)=\alpha(x,u,du)+\tilde\beta_0(x)\partial_{\mathbf{n}}u
$$
and the boundary equation \eqref{1.2-AN95} takes on the form
\begin{align}\label{VP2}
-\tilde{\alpha}^{ij}(x)d_id_ju
+& \tilde\beta^{*i}(x)d_iu+\tilde\beta_0(x)\partial_{\mathbf{n}}u\\
\nonumber
+& \mu\tilde{\alpha}(x)|du(x)|^2 +\Theta(x) u=\tilde{g}(x)
\quad \text{a.e. on}\ \partial\Omega,
\end{align}
where $\tilde{g}(x):= \Theta(x) \big(u(x)-\tilde\alpha(x)\big) \in L^{n-1}(\partial\Omega).$ 
Having in mind \eqref{V4} and
without loss of generality we may suppose hereafter that $\Theta\ge 1$.

Consider now 
the one-parameter family of Venttsel problems
\begin{align}
\label{VP1-delta}
& -\tilde{a}^{ij} D_{i}D_jv_\delta + \tilde{b}^i D_iv_\delta + \mu\tilde{a}|Dv_\delta|^2+\Phi v_\delta=\delta\tilde{f} & \text{a.e. in}\ \Omega,\\
\label{VP2-delta}
& -\tilde{\alpha}^{ij}d_{i}d_jv_\delta + \tilde\beta^{*i} d_iv_\delta+\tilde\beta_0\partial_{\mathbf{n}}v_\delta
+ \mu\tilde{\alpha}|dv_\delta|^2+\Theta v_\delta
=\delta\tilde{g} &\text{a.e. on}\ \partial\Omega.
\end{align}

Following the strategy of \cite{AC78}, we will prove unique solvability of  \eqref{VP1-delta}--\eqref{VP2-delta}  in $V_{n,n-1}(\Omega)$
 for each $\delta\in[0,1]$ and will estimate the gradient of $v_{\delta_2}$ in terms of  the gradient of
$v_{\delta_1}$ for small enough $\delta_2-\delta_1>0$. Then we will easily have $v_0\equiv0$, while the coincidence of the problem \eqref{VP1-delta}--\eqref{VP2-delta} with \eqref{VP1}--\eqref{VP2} for $\delta=1$  would give $v_1=u$, and finite iteration in $\delta$ will  give the desired bound \eqref{grad-est}. 

To realize that plan, we need two lemmata.
\begin{lemma}
\label{lemma-L-infty}
Suppose that $v_{\delta_1},v_{\delta_2}\in V_{n,n-1}(\Omega)$ solve \eqref{VP1-delta}--\eqref{VP2-delta} with $\delta_1\le\delta_2$.

 Then
\begin{equation}
\label{36}
\|v_{\delta_1}-v_{\delta_2}\|_{\infty,\Omega}
\le (\delta_2-\delta_1)\big(M_0+{\eta}(M_0)\big).
\end{equation}
\end{lemma}
\begin{proof}
Setting $w=v_{\delta_1}-v_{\delta_2}$ we obtain the following \textit{linear} Venttsel problem
\begin{align*}
&\tilde{\mathcal{L}}w:= -\tilde{a}^{ij}D_iD_jw + (\tilde{a}^i+\tilde{b}^i)D_iw+\Phi w= (\delta_1-\delta_2)\tilde{f} &\text{a.e. in}\ \Omega,\\
&\tilde{\mathcal{B}}w:= -\tilde{\alpha}^{ij}d_id_jw +
(\tilde{\alpha}^i+\tilde\beta^{*i})d_iw +\tilde\beta_0\partial_{\mathbf{n}}w+\Theta w =(\delta_1-\delta_2)\tilde{g} &\text{a.e. on}\ \partial\Omega,
\end{align*}
with 
\begin{align*}
\tilde{a}^i(x):=&\ 2\mu\tilde{a}(x)\int_0^1\big(tD_iw(x)+D_iv_{\delta_2}(x)\big)dt,\\
\tilde{\alpha}^i(x):=&\ 2\mu\tilde{\alpha}(x)\int_0^1\big(td_iw(x)+d_iv_{\delta_2}(x)\big)dt.
\end{align*}
We recall that $|\tilde{a}|\le{\eta}(M_0)$, $w,v_{\delta_2}\in W^2_n(\Omega)$ and therefore $\tilde{a}^i\in L^n(\Omega)$. Similarly, $\tilde{\alpha}^i\in L^{n-1}(\partial\Omega)$.

Using $\tilde{f}(x)\ge -\Phi(x)\big(M_0+{\eta}(M_0)\big)$ we get
$$
\tilde{\mathcal{L}}w\le 
\tilde{\mathcal{L}}\big((\delta_2-\delta_1)(M_0+{\eta}(M_0))\big)
\qquad \text{a.e. in}\ \Omega
$$
and similarly
$$
\tilde{\mathcal{B}}w\le 
\tilde{\mathcal{B}}\big((\delta_2-\delta_1)(M_0+{\eta}(M_0))\big)\qquad \text{a.e. on}\ \partial\Omega.
$$
It follows from Theorem~\ref{global-max-principle} that
$$
w(x)\le (\delta_2-\delta_1)\big(M_0+{\eta}(M_0)\big)\qquad \text{in}\ \Omega.
$$
In the same manner the lower estimate $w(x)\ge -(\delta_2-\delta_1)\big(M_0+{\eta}(M_0)\big)$ follows and this gives the claim \eqref{36}.
\end{proof}

It is worth noting that setting $\delta_1=\delta_2$ in \eqref{36}, we get immediately $v_{\delta_1}\equiv v_{\delta_2}$ and thus uniqueness of solutions to \eqref{VP1-delta}--\eqref{VP2-delta}. 
Precisely,
\begin{cor}
\label{cor-uniqueness}
The problem \eqref{VP1-delta}--\eqref{VP2-delta} cannot have more than one solution in $V_{n,n-1}(\Omega)$ for any $\delta\in[0,1]$.
\end{cor}

\begin{lemma}\label{lemma-grad}
Under the hypotheses of Lemma $\ref{lemma-L-infty}$, there is a $\varkappa>0$ such that the inequality $\delta_2-\delta_1\le\varkappa$ implies
\begin{align}
\label{39}
\|Dv_{\delta_2}-Dv_{\delta_1}\|_{2n,\Omega} 
+& \|dv_{\delta_2}-dv_{\delta_1}\|_{2(n-1),\partial\Omega}\\ 
&\le\ C_2 (\delta_2-\delta_1) \Big(1+\|Dv_{\delta_1}\|_{2n,\Omega} + \|dv_{\delta_1}\|_{2(n-1),\partial\Omega}\Big).
\nonumber
\end{align}
The constants $\varkappa$ and $C_2$ depend on the same quantities as $M_1$ in the statement of Theorem~$\ref{Du-est}.$
\end{lemma}
\begin{proof} We rewrite the problem for
$w=v_{\delta_1}-v_{\delta_2}$ as follows:
\begin{align}
\label{hat-L}
\widehat{\mathcal{L}}w:=&\
-\tilde{a}^{ij}D_iD_jw + \tilde{b}^iD_iw+\Phi w= \widehat{f} &\text{a.e. in}\ \Omega,\\
\label{hat-B}
\widehat{\mathcal{B}}w:=&\
 -\tilde{\alpha}^{ij}d_id_jw +
\tilde\beta^{*i}d_iw +\tilde\beta_0\partial_{\mathbf{n}}w+\Theta w =\widehat{g} &\text{a.e. on}\ \partial\Omega,
\end{align}
with 
\begin{align*}
\widehat f= &\ (\delta_1-\delta_2)\tilde{f}-\mu
\tilde{a}\big(|Dv_{\delta_1}|^2-|Dv_{\delta_2}|^2\big)\in L^n(\Omega),\\
\widehat g = &\ 
(\delta_1-\delta_2)\tilde{g}-\mu
\tilde{\alpha}\big(|dv_{\delta_1}|^2-|dv_{\delta_2}|^2\big)\in L^{n-1}(\partial\Omega).
\end{align*}
Theorem~\ref{apriori-estimate} and Lemma \ref{lemma-L-infty} yield
\begin{equation}
\label{37}
\|w\|_{V_{n,n-1}(\Omega)}\le N_9 \left( 
\big\|\widehat{f}\big\|_{n,\Omega}+
\big\|\widehat{g}\big\|_{n-1,\partial\Omega}+(\delta_2-\delta_1)\big(M_0+{\eta}(M_0)\big)\right),
\end{equation}
where $N_9$ depends only on $n$, $\nu$, $\mathrm{diam}\,\Omega$, the properties of $\partial\Omega$, $M_0$, the norms $\|\Phi\|_{n,\Omega}$ and $\|\Theta\|_{n-1,\partial\Omega}$, the moduli of continuity of the functions $b$, $\beta_0$ and $\beta$ in the corresponding Orlicz spaces defined by conditions \eqref{A5}, \eqref{V0} and \eqref{V5}, respectively, and on the \textit{VMO}-moduli of the coefficients $\tilde{a}^{ij}(x)$  and $\tilde{\alpha}^{ij}(x)$. 

However, as explained before, the 
 \textit{VMO}-moduli of $\tilde{a}^{ij}(x)$  and $\tilde{\alpha}^{ij}(x)$ are controlled in terms of $M_0$ through \cite[Lemma~2.1]{P95} and Proposition~\ref{Holder-est}. Thus, the constant $N_9$ in \eqref{37} depends only on data listed in the statement of Lemma~\ref{lemma-grad}.

Taking advantage of the bounds
\begin{align*}
\|\tilde{f}\|_{n,\Omega}\le  \|\Phi\|_{n,\Omega}\big(M_0+{\eta}(M_0)\big),\qquad
\|\tilde{g}\|_{n-1,\partial\Omega}\le  \|\Theta\|_{n-1,\partial\Omega}\big(M_0+{\eta}(M_0)\big),
\end{align*}
and of the evident inequalities
\begin{align*}
\left\| |Dv_{\delta_1}|^2-|Dv_{\delta_2}|^2\right\|_{n,\Omega}\le &\ \|Dw\|^2_{2n,\Omega}+2\|Dw\|_{2n,\Omega}\|Dv_{\delta_1}\|_{2n,\Omega},\\
\left\| |dv_{\delta_1}|^2-|dv_{\delta_2}|^2\right\|_{n-1,\partial\Omega}\le &\ \|dw\|^2_{2(n-1),\partial\Omega}+2\|dw\|_{2(n-1),\partial\Omega}\|dv_{\delta_1}\|_{2(n-1),\partial\Omega},
\end{align*}
we rewrite \eqref{37} as follows
\begin{align}
\label{38}
\|w\|_{V_{n,n-1}(\Omega)}\le &\ N_{10} \Big( (\delta_2-\delta_1)+
\|Dw\|^2_{2n,\Omega}+\|dw\|^2_{2(n-1),\partial\Omega}\\
\nonumber
 &+ \|Dw\|_{2n,\Omega}\|Dv_{\delta_1}\|_{2n,\Omega} +\|dw\|_{2(n-1),\partial\Omega}\|dv_{\delta_1}\|_{2(n-1),\partial\Omega}\Big),
\end{align}
where $N_{10}$ depends on the same quantities as $N_9$. 

We infer now the Gagliardo--Nirenberg interpolation inequality \cite[Theorem 15.1]{BIN75}
and the estimate \eqref{36} to get
\begin{align}
\label{Dw}
\|Dw\|^2_{2n,\Omega}\le&\ C(n,\Omega)\big(\|D^2w\|_{n,\Omega}+\|w\|_{\infty,\Omega}\big)\|w\|_{\infty,\Omega}\\
\nonumber
\le &\ C(n,\Omega)(\delta_2-\delta_1)\big(M_0+{\eta}(M_0)\big)\|w\|_{V_{n,n-1}(\Omega)},
\end{align}
and similarly
\begin{equation}\label{dw}
\|dw\|^2_{2(n-1),\partial\Omega}\le
C(n,\Omega)(\delta_2-\delta_1)\big(M_0+{\eta}(M_0)\big)\|w\|_{V_{n,n-1}(\Omega)}.
\end{equation}
We substitute these inequalities into \eqref{38} and estimate the last two terms by the Cauchy inequality. This gives
\begin{align*}
\|w\|_{V_{n,n-1}(\Omega)}\le &\ N_{11} \Big( (\delta_2-\delta_1)+(\delta_2-\delta_1+\varkappa)\|w\|_{V_{n,n-1}(\Omega)}\\
&+ \frac{\delta_2-\delta_1}{\varkappa}
\big(\|Dv_{\delta_1}\|^2_{2n,\Omega}+\|dv_{\delta_1}\|^2_{2(n-1),\partial\Omega}\big)\Big)
\end{align*}
with arbitrary $\varkappa\in(0,1)$ and where
$N_{11}$ depends on the same quantities as $N_9$. 
Choosing $\varkappa=\frac 1{4N_{11}}$ we obtain \eqref{39} for $\delta_2-\delta_1\le\varkappa$, in view of \eqref{Dw} and \eqref{dw}.
\end{proof}

\medskip

Turning back to
the proof of Theorem \ref{Du-est}, we fix $\delta_1=0$ and $\delta_2=\varkappa$ in \eqref{39} and remember that $v_0\equiv0$ by Corollary~\ref{cor-uniqueness}. This gives the \textit{a~priori} estimate
\begin{equation}
\label{40}
\|Dv_{\varkappa}\|_{2n,\Omega}+\|dv_{\varkappa}\|_{2(n-1),\partial\Omega} \le C_2\varkappa.
\end{equation}

The solvability of \eqref{VP1-delta}--\eqref{VP2-delta} with $\delta=\varkappa$ is a consequence of the Leray--Schauder fixed point theorem. Indeed, define the nonlinear operator 
$$\mathcal{F}: \ W^1_{2n}(\Omega)\cap W^1_{2(n-1)}(\partial\Omega) \mapsto V_{n,n-1}(\Omega)
$$ 
which associates to any $w\in W^1_{2n}(\Omega)\cap W^1_{2(n-1)}(\partial\Omega)$ the unique solution $v=\mathcal{F}(w)$ of the \textit{linear} Venttsel problem
$$
\begin{cases}
\widehat{\mathcal{L}}v=  \varkappa \tilde{f}(x)-
\tilde{a}(x)|Dw|^2 &\text{a.e. in}\ \Omega, \\[4pt]
\widehat{\mathcal{B}}v
 =\varkappa\tilde{g}(x)-
\tilde{\alpha}(x)|dw|^2 &\text{a.e. on}\ \partial\Omega,
\end{cases}
$$
with operators $\widehat{\mathcal{L}}$ and $\widehat{\mathcal{B}}$ given by \eqref{hat-L} and \eqref{hat-B}, respectively.

The unique solvability of that problem follows from Theorems~\ref{existence} and \ref{global-max-principle} 
due to assumptions of Theorem \ref{Du-est} and $w\in W^1_{2n}(\Omega)\cap W^1_{2(n-1)}(\partial\Omega)$ (recall that $\Theta\ge 1$). Therefore, the nonlinear operator $\mathcal{F}$ is well defined. Moreover, the problem \eqref{VP1-delta}--\eqref{VP2-delta} with $\delta=\varkappa$ is equivalent to the equation $u=\mathcal{F}(u)$. 

The estimate \eqref{2.3-AN95'} yields the continuity of $\mathcal{F}$, while the compactness of the embedding $V_{n,n-1}(\Omega) \hookrightarrow W^1_{2n}(\Omega)\cap W^1_{2(n-1)}(\partial\Omega)$ guarantees the compactness of $\mathcal{F}$ considered as a mapping from $W^1_{2n}(\Omega)\cap W^1_{2(n-1)}(\partial\Omega)$ into itself. Finally, any solution of the equation $v=\sigma\mathcal{F}(v)$, $0\le\sigma\le1$, that is,
$$
\begin{cases}
\widehat{\mathcal{L}}v=  \sigma\big(\varkappa \tilde{f}(x)-
\tilde{a}(x)|Dv|^2\big) &\text{a.e. in}\ \Omega, \\[4pt]
\widehat{\mathcal{B}}v
 =\sigma\big(\varkappa\tilde{g}(x)-
\tilde{\alpha}(x)|dv|^2\big) &\text{a.e. on}\ \partial\Omega,
\end{cases}
$$
satisfies, by \eqref{40}, the \textit{a~priori} estimate
$$
\|Dv\|_{2n,\Omega}+\|dv\|_{2(n-1),\partial\Omega} \le C_3
$$
with $C_3=C_2\varkappa$ independent of $\sigma$. This suffices to combine the Leray--Schauder theorem (see, e.g., \cite[Theorem 11.6]{GT01}) with Corollary~\ref{cor-uniqueness} in order to get unique solvability of \eqref{VP1-delta}--\eqref{VP2-delta} with $\delta=\varkappa$.

To complete the proof of Theorem~\ref{Du-est}, we take  successively $\delta_1=k\varkappa$, $\delta_2=(k+1)\varkappa$, $k\in\mathbb{N}$, and repeat the above procedure. Finitely many iterations of \eqref{39} lead to \eqref{grad-est}
since $v_1$ is nothing else than the solution $u$ of the problem \eqref{VP1}, \eqref{VP2}.
\end{proof}

Based on the \textit{a~priori} gradient estimate derived in Theorem~\ref{Du-est}, we can get solvability of the quasilinear Venttsel problem
\eqref{1.1-AN95}--\eqref{1.2-AN95}
under the hypotheses listed at the beginning of Section~\ref{sec4}.

\begin{thm}
\label{quasilinear-existence}
Let $\partial\Omega \in \C^{1,1}$ and let the functions involved in \eqref{1.1-AN95}--\eqref{1.2-AN95} satisfy 
the conditions \eqref{A1}--\eqref{A5}, \eqref{V}, \eqref{V0}--\eqref{V5}.

If any solution to the one-parameter family of Venttsel problems
\begin{align}
\label{VP1-sigma}
& -a^{ij}(x,u)D_{i}D_ju+ \sigma a(x,u,Du)=0 & \text{a.e. in}\ \Omega,\\
\label{VP2-sigma}
& -\alpha^{ij}(x,u)d_{i}d_ju+(1-\sigma) u +
\sigma\alpha(x,u,Du)
=0 & \text{a.e. on}\ \partial\Omega
\end{align}
satisfies the \textit{a~priori} estimate
\begin{equation}\label{L-infty}
    \sup_{\Omega} |u|\le M_0
\end{equation}
with $M_0$ independent of $u$ and $\sigma\in[0,1]$, then the quasilinear Venttsel problem \eqref{1.1-AN95}--\eqref{1.2-AN95} is solvable in the space $V_{n,n-1}(\Omega)$.
\end{thm}

\begin{proof}
We again proceed by using the Leray--Schauder theorem. Introduce the linear space
$$
U_{n,n-1}(\Omega):=\Big\{v\in W^1_{2n}(\Omega)\cap W^1_{2(n-1)}(\partial\Omega)\colon\quad \partial_{\mathbf{n}}v\in L^{1}(\partial\Omega)\Big\},
$$
equipped with the natural norm
$$
\|v\|_{U_{n,n-1}(\Omega)}=
\|v\|_{W^1_{2n}(\Omega)}+
\|v\|_{W^1_{2(n-1)}(\partial\Omega)}+
\|\partial_{\mathbf{n}}v\|_{L^{1}(\partial\Omega)}.
$$
and define the nonlinear operator 
$$\mathcal{F}_1\colon\ \ U_{n,n-1}(\Omega) \mapsto V_{n,n-1}(\Omega)
$$ 
which associates to any  $u\in U_{n,n-1}(\Omega)$ the solution $v=\mathcal{F}_1(u)$ of the \textit{linear} Venttsel problem
\begin{align}
\label{bb-L}
{\mathbb L} v:=&\
-\tilde{a}^{ij}D_iD_jv + \tilde{b}^iD_iv = -\tilde{a}(\mu|Du|^2+\Phi) &\text{a.e. in}\ \Omega,\\
\label{bb-B}
{\mathbb B}v:=&\
 -\tilde{\alpha}^{ij}d_id_jv +
\tilde\beta^{*i}d_iv + \tilde\beta_0\partial_{\mathbf{n}}v+v\\ 
\nonumber
& \qquad\qquad\qquad\qquad = -\tilde\alpha(\mu|du|^2+\Theta) +u &\text{a.e. on}\ \partial\Omega,
\end{align}
where $\tilde{a}^{ij}$, $\tilde{a}$ and $\tilde{b}^i$ are defined by  \eqref{tilde-aij}--\eqref{tilde-b}, while $\tilde{\alpha}^{ij}$, $\tilde{\alpha}$, $\tilde\beta^{*i}$ and $\tilde\beta_0$ are given by \eqref{tilde-alpij}--\eqref{beta0u}, respectively.

By the Morrey embedding theorem, see \cite[Theorem 7.17]{GT01}, we have $u\in \mathcal{C}^{0,1/2}(\overline{\Omega})$.  
Similarly to the proof of Theorem \ref{Du-est}, we establish that the problem \eqref{bb-L}--\eqref{bb-B} satisfies all the hypotheses of Theorems~\ref{apriori-estimate} and \ref{global-max-principle}.  On the base of
Theorems \ref{existence} and \ref{global-max-principle} we conclude that the problem \eqref{bb-L}--\eqref{bb-B} is uniquely solvable in $V_{n,n-1}(\Omega)$ and therefore the nonlinear operator $\mathcal{F}_1$ is well defined. Moreover, it is easy to see that the original quasilinear Venttsel problem \eqref{1.1-AN95}--\eqref{1.2-AN95} is equivalent to the equation $u=\mathcal{F}_1(u)$. 

Since $u\in V_{n,n-1}(\Omega)$ implies $\partial_\mathbf{n}u\in W^1_{n}(\Omega)
\subset L^r(\partial\Omega)$ for all $r<\infty$, see, e.g., \cite[Sec.~10.5]{BIN75}, the space $V_{n,n-1}(\Omega)$ is embedded into $U_{n,n-1}(\Omega)$ and the embedding
is compact. This guarantees the compactness of $\mathcal{F}_1$ considered as a mapping from $U_{n,n-1}(\Omega)$ into itself.

To prove the continuity of $\mathcal{F}_1$ in the space $U_{n,n-1}(\Omega)$,
we consider a sequence $u_h\in U_{n,n-1}(\Omega)$ such that $u_h\to u$ in $U_{n,n-1}(\Omega)$ as $h\to \infty$, and set $v_h=\mathcal{F}_1(u_h)$, $v=\mathcal{F}_1(u)$. Thus, $v$ is the solution of \eqref{bb-L}--\eqref{bb-B} while $v_h$ solves the problem
\begin{align}
\label{bb-Lh}
{\mathbb L}_h v_h:=&\ -\tilde{a}_h^{ij}D_iD_jv_h + \tilde{b}_h^iD_iv_h = -\tilde{a}_h(\mu|Du_h|^2+\Phi) &\text{a.e. in}\ \Omega,\\
\label{bb-Bh}
{\mathbb B}_h v_h:=&\ -\tilde{\alpha}_h^{ij}d_id_jv_h +
\tilde\beta^{*i}_hd_iv_h + \tilde\beta_{0,h}\partial_{\mathbf{n}}v_h+v_h\\ 
\nonumber
& \qquad\qquad\qquad\qquad = -\tilde\alpha_h(\mu|du_h|^2+\Theta) +u_h &\text{a.e. on}\ \partial\Omega,
\end{align}
where $\tilde{a}_h^{ij}$, $\tilde{a}_h$, $\tilde{b}_h^i$,  $\tilde{\alpha}_h^{ij}$, $\tilde{\alpha}_h$, $\tilde\beta^{*i}_h$ and $\tilde\beta_{0,h}$ are defined similarly to \eqref{tilde-aij}--\eqref{tilde-b} and \eqref{tilde-alpij}--\eqref{beta0u} with $u$ replaced by $u_h$.

Notice that $\|u_h\|_{\mathcal{C}^{0,1/2}(\overline{\Omega})}$ is uniformly bounded by the Morrey theorem. So, by
\cite[Lemma~2.1]{P95} the 
 \textit{VMO}-moduli of $\tilde{a}_h^{ij}(x)$  and $\tilde{\alpha}_h^{ij}(x)$ are uniformly bounded. Further, by definition we have
 $$
|\tilde{b}_h^i(x)|\le {\eta}(M)b(x),\quad 
|\tilde\beta^{*i}_h(x)|\le {\eta}(M)\beta(x),\quad
0\le\tilde\beta_{0,h}(x)\le{\eta}(M)\beta_0(x),
$$
where $M=\sup_h\|u_h\|_{\mathcal{C}(\overline{\Omega})}$ while
the functions $b$, $\beta$ and $\beta_0$ verify the assumptions \eqref{L3}, \eqref{B3} and \eqref{B5}, respectively. It follows from 
Remark~\ref{rem1} that any solution of the problem
\begin{equation} \label{thick-h-problem}
{\mathbb L}_h w_h={\mathbb F}_h \quad \text{a.e. in}\ \Omega,\qquad
  {\mathbb B}_h w_h={\mathbb G}_h \quad \text{a.e. on}\ \partial\Omega
\end{equation}
satisfies the  estimate 
\begin{equation} 
\label{2.3h-AN95}
\|w_h\|_{V_{n,n-1}(\Omega)} \le C_4 \Big( 
\|{\mathbb F}_h\|_{n,\Omega}+\|{\mathbb G}_h\|_{n-1,\partial\Omega}+
\|w_h\|_{n,\Omega}+\|w_h\|_{n-1,\partial\Omega}\Big),
\end{equation}
where the constant $C_4$ is independent of $h$.

As in Corollary~\ref{corollary},  the lower-order terms can be dropped from the right-hand side of \eqref{2.3h-AN95}.
\begin{lemma}
\label{dropped-est-h}
For any solution of the problem \eqref{thick-h-problem},
the following estimate holds
\begin{equation} 
\label{2.3h-AN95'}
\|w_h\|_{V_{n,n-1}(\Omega)} \le \widetilde C_4 \big(\|{\mathbb F}_h\|_{n,\Omega}+\|{\mathbb G}_h\|_{n-1, \partial\Omega} \big)
\end{equation}
with a constant $\widetilde C_4$ independent of $h$ and $w_h$.
\end{lemma}
\begin{proof}
As in the proof of Corollary~\ref{corollary}, we proceed by contradiction. Assume that the statement is false. Then there is a sequence 
$$
{\mathbb F}_h\to0 \quad \text{in} \ L^n(\Omega), \qquad {\mathbb G}_h\to0 \quad \text{in} \ L^{n-1}(\partial\Omega),
$$
such that the corresponding solutions of \eqref{thick-h-problem} are bounded away from zero in $V_{n,n-1}(\Omega)$. Without loss of generality we may suppose that $\|w_h\|_{V_{n,n-1}(\Omega)}=1$. Then, passing if necessary to a subsequence, one has
$$
w_h\rightharpoonup\ w \quad \text{in} \ V_{n,n-1}(\Omega), \quad\ 
w_h\to w \quad \text{in} \ L^n(\Omega), \quad\  w_h\to w \quad \text{in} \ L^{n-1}(\partial\Omega).
$$
Since the difference $w_{h,k}:=w_h-w_k$ solves the problem
\begin{align*}
{\mathbb L}_h w_{h,k}=
\big(\tilde{a}_h^{ij}-\tilde{a}_k^{ij}\big)D_iD_jw_k 
 &\ + \big(\tilde{b}_k^i-\tilde{b}_h^i\big)D_iw_k 
+ {\mathbb F}_h - {\mathbb F}_k\quad \text{a.e. in}\ \Omega,
\\
{\mathbb B}_h w_{h,k}= \big(\tilde{\alpha}_h^{ij}-\tilde{\alpha}_k^{ij}\big)
d_id_jw_k  &\ + 
\big(\tilde\beta^{*i}_k-\tilde\beta^{*i}_h\big)d_iw_k \\
 &\ +  \big(\tilde\beta_{0,k}-\tilde\beta_{0,h}\big)\partial_{\mathbf{n}}w_k + {\mathbb G}_h - {\mathbb G}_k\quad \text{a.e. on}\ \partial\Omega,
\end{align*}
the estimate \eqref{2.3h-AN95} yields
\begin{align}
\label{2.3hk-AN95}
&\|w_{h,k}\|_{V_{n,n-1}(\Omega)}\le C_4\Big(\left\|\big(\tilde{a}_h^{ij}-\tilde{a}_k^{ij}\big)D_{i}D_jw_k \right\|_{n,\Omega}
\\
\nonumber
 &\qquad\qquad + \left\|\big(\tilde{\alpha}_h^{ij}-\tilde{\alpha}_k^{ij}\big)
d_id_jw_k\right\|_{n-1,\partial\Omega} +\!
\big\|\big(\tilde{b}_k^i-\tilde{b}_h^i\big)d_iw_k\big\|_{n,\Omega}  \\
\nonumber
 &\qquad\qquad +  \big\|\big(\tilde\beta^{*i}_k-\tilde\beta^{*i}_h\big)d_iw_k\big\|_{n-1,\partial\Omega} +\! \big\|\big(\tilde\beta_{0,k}-\tilde\beta_{0,h}\big)\partial_{\mathbf{n}}w_k\big\|_{n-1,\partial\Omega} \\
\nonumber
 &\qquad\qquad +  \|{\mathbb F}_h - {\mathbb F}_k\|_{n,\Omega} + \|{\mathbb G}_h - {\mathbb G}_k\|_{n-1,\partial\Omega} \\
\nonumber
 &\qquad\qquad +  \|w_{h,k}\|_{n,\Omega} + \|w_{h,k}\|_{n-1,\partial\Omega}\Big).
\end{align}
We know that $\|D_{i}D_jw_k\|_{n,\Omega}$ are bounded, while 
$$
\left\|\tilde{a}_h^{ij}-\tilde{a}_k^{ij}\right\|_{\infty,\Omega}=\left\|a^{ij}(\cdot,u_h)-a^{ij}(\cdot,u_k)\right\|_{\infty,\Omega}\to0
$$
as $h,k\to\infty$, as consequence of $u_h\to u$ in $\mathcal{C}^{0,1/2}(\overline{\Omega})$ and the hypothesis \eqref{A3}. 

Therefore, the first term 
on the right-hand side of \eqref{2.3hk-AN95} tends to $0$ as $h,k\to \infty$. The second term is managed in the same way.

Further on, 
$$
\tilde\beta^{*i}_k-\tilde\beta^{*i}_h=\beta\,\bigg(\dfrac{\alpha(\cdot,u_k,du_k)}{\mu|du_k|^2+\beta|du_k|+\Theta}\, \dfrac{d_iu_k}{|du_k|}-\dfrac{\alpha(\cdot,u_h,du_h)}{\mu|du_h|^2+\beta|du_h|+\Theta}\, \dfrac{d_iu_h}{|du_h|}\bigg)
$$
tends to zero a.e. on $\partial\Omega$ as $h,k\to\infty$. Since $\big|\tilde\beta^{*i}_k(x)-\tilde\beta^{*i}_h(x)\big|\le 2{\eta}(M)\beta(x)$, the hypothesis \eqref{V5} together with the Lebesgue dominated convergence theorem ensure that
$$
\int_{\partial\Omega}
\big|\tilde\beta^{*i}_k(x)-\tilde\beta^{*i}_h(x)\big|^{n-1}
\left(\log{\left(1+\big|\tilde\beta^{*i}_k(x)-\tilde\beta^{*i}_h(x)\big|^{n-1}\right)}\right)^{1-1/(n-1)}dx \to0
$$
as $h,k\to\infty$. This gives
$$
\big\|\big(\tilde\beta^{*i}_k-\tilde\beta^{*i}_h\big)d_iw_k\big\|_{n-1,\partial\Omega}\le 
\big\||\tilde{\boldsymbol{\beta}}{}^*_k-\tilde{\boldsymbol{\beta}}{}^*_h|^{n-1}\big\|_{L_{\Psi^*}(\partial\Omega)}
\big\||dw_k|^{n-1}\big\|_{L_{\Psi}(\partial\Omega)} \to0
$$
as $h,k\to\infty$ (recall that $L_{\Psi^*}(\partial\Omega)$ is the Orlicz space dual to $L_{\Psi}(\partial\Omega)$ introduced in \eqref{Orlicz}).
The third and the fifth terms on the right-hand side of \eqref{2.3hk-AN95} are managed in the same manner. The last four terms evidently tend to zero and we conclude that $w_h\to w$ in $V_{n,n-1}(\Omega)$. In particular, $\|w\|_{V_{n,n-1}(\Omega)}=1$.

Finally, passing to the limit in \eqref{thick-h-problem} yields
$$
{\mathbb L} w=0 \quad \text{a.e. in}\ \Omega,\qquad {\mathbb B}w=0 \quad \text{a.e. on}\ \partial\Omega,
$$
whence $w=0$ by Theorem~\ref{global-max-principle}.   The contradiction obtained completes the proof.
\end{proof}

Turning back to the proof of Theorem \ref{quasilinear-existence},  we have
\begin{equation*}
\|v_h\|_{V_{n,n-1}(\Omega)} \le \widetilde C_4{\eta}(M) \Big( \left\|\mu|Du_h|^2+\Phi\right\|_{n,\Omega}
+\left\|\mu|du_h|^2+\Theta +|u_h|\right\|_{n-1,\partial\Omega} \Big)    
\end{equation*}
as consequence of Lemma~\ref{dropped-est-h} (recall that $|\tilde a_h(x)|\le {\eta}(M)$ by \eqref{A4} and
$|\tilde \alpha_h(x)|\le {\eta}(M)$ by \eqref{V4}).
The  terms on the right-hand side above are uniformly bounded because of the boundedness of $u_h$ in $U_{n,n-1}(\Omega)$.
This means that the sequence $v_h$ is bounded in $V_{n,n-1}(\Omega)$. 

The difference $v-v_h$ solves the problem
\begin{align*}
{\mathbb L} (v-v_h)=&\ 
\big(\tilde{a}^{ij}-\tilde{a}_h^{ij}\big)D_iD_jv_h 
+ \big(\tilde{b}_h^i-\tilde{b}^i\big)D_iv_h \\
&\quad +  \tilde{a}_h\big(\mu|Du_h|^2+\Phi\big) - \tilde{a}\big(\mu|Du|^2+\Phi\big)\quad \text{a.e. in}\ \Omega,
\\
{\mathbb B} (v-v_h)=&\ \big(\tilde{\alpha}^{ij}-\tilde{\alpha}_h^{ij}\big)
d_id_jv_h +
\big(\tilde\beta^{*i}_h-\tilde\beta^{*i}\big)d_iv_h \\
&\quad +   \big(\tilde\beta_{0,h}-\tilde\beta_0\big)\partial_{\mathbf{n}}v_h  + (u - u_h)\\ 
&\quad +  \tilde\alpha_h\big(\mu|du_h|^2+\Theta\big) - \tilde\alpha\big(\mu|du|^2+\Theta\big)
\quad \text{a.e. on}\ \partial\Omega
\end{align*}
and the estimate \eqref{2.3-AN95'} yields
\begin{align*}
\|v-v_h\|_{V_{n,n-1}(\Omega)} \le&\  \widetilde C_1\Big(\left\|\big(\tilde{a}^{ij}-\tilde{a}_h^{ij}\big)D_{i}D_jv_h  \right\|_{n,\Omega} +
\big\|\big(\tilde{b}_h^i-\tilde{b}^i\big)d_iv_h\big\|_{n,\Omega}
\\
 &+  \left\|\big(\tilde{\alpha}^{ij}-\tilde{\alpha}_h^{ij}\big)
d_id_jv_h\right\|_{n-1,\partial\Omega} + \big\|\big(\tilde\beta^{*i}_h-\tilde\beta^{*i}\big)d_iv_h\big\|_{n-1,\partial\Omega}  \\
 &+  \big\|\big(\tilde\beta_{0,h}-\tilde\beta_0\big)\partial_{\mathbf{n}}v_h\big\|_{n-1,\partial\Omega} + \|u-u_h\|_{n-1,\partial\Omega}\\
 &+ \mu\left\|\tilde{a}_h|Du_h|^2\! -\! \tilde{a}|Du|^2\right\|_{n,\Omega} + \mu\left\|\tilde\alpha_h|du_h|^2\! -\! \tilde\alpha|du|^2\right\|_{n-1,\partial\Omega}\\
 &+  \left\|\big(\tilde{a}_h-\tilde{a}\big)\Phi\right\|_{n,\Omega} +\left\|\big(\tilde\alpha_h-\tilde\alpha\big)\Theta\right\|_{n-1,\partial\Omega}\Big) .
\end{align*}
We estimate the first five terms on the right-hand side  above  in the same way as done for the corresponding terms in \eqref{2.3hk-AN95}. Further on,
\begin{align*}
\left\|\tilde{a}_h|Du_h|^2 - \tilde{a}|Du|^2\right\|_{n,\Omega}\le &\ 
\|\tilde{a}_h\|_{\infty,\Omega}\left\||Du_h|^2 - |Du|^2\right\|_{n,\Omega}\\
&\ + \left\|\big(\tilde{a}_h- \tilde{a}\big) |Du|^2\right\|_{n,\Omega}.
\end{align*}
The first term tends to zero since $u_h\to u$ in $U_{n,n-1}(\Omega)$, while the second one is infinitesimal by the Lebesgue theorem. All the remaining terms are estimated in a similar way and we obtain finally $v_h\to v$ in $V_{n,n-1}(\Omega)$  that proves the continuity of the operator $\mathcal{F}_1$.

\medskip

In order to apply the Leray--Schauder theorem and to get existence of a fixed point of $\mathcal{F}_1$, we present a family of continuous, compact nonlinear operators $\mathcal{T}(\cdot, \sigma)$ continuously depending on the parameter $\sigma\in[0,1]$ such that 
$$
\mathcal{T}(u, 0)\equiv0;\qquad \mathcal{T}(u, 1)=\mathcal{F}_1(u), \qquad u\in U_{n,n-1}(\Omega).
$$
Namely, given a $\sigma\in[0,1]$, the operator $\mathcal{T}(\cdot, \sigma)$
associates to any  $u\in U_{n,n-1}(\Omega)$ the unique solution $v_\sigma$ of the linear Venttsel problem
\begin{align*}
&-\tilde{a}^{ij}D_iD_jv_\sigma + \sigma\tilde{b}^iD_iv_\sigma = -\sigma\tilde{a}\big(\mu|Du|^2+\Phi\big)
&\text{a.e. in}\ \Omega,\\
&-\tilde{\alpha}^{ij}d_id_jv_\sigma +
\sigma\tilde\beta^{*i}d_iv_\sigma +  \sigma\tilde\beta_0\partial_{\mathbf{n}}v_\sigma+v_\sigma\\ 
&\qquad\qquad\qquad\qquad\qquad\ = -\sigma\tilde\alpha\big(\mu|du|^2+\Theta\big) +\sigma u
&\text{a.e. on}\ \partial\Omega.
\end{align*}

All mentioned properties of the family $\mathcal{T}(\cdot, \sigma)$ follow from the previous arguments, and to apply the Leray--Schauder theorem it remains only to derive the \textit{a~priori} estimate
\begin{equation}
\label{AA}
\|u\|_{U_{n,n-1}(\Omega)} \le C_5
\end{equation}
 for any solution of the equation $u =\mathcal{T}(u, \sigma)$ in $U_{n,n-1}(\Omega)$, with a constant $C_5$ independent of $\sigma\in[0,1]$ and $u$.

To this end, we notice that the equation $u =\mathcal{T}(u, \sigma)$
  is equivalent to the Venttsel problem \eqref{VP1-sigma}--\eqref{VP2-sigma}. We apply 
Theorem \ref{Du-est} and take into account that
under the assumption \eqref{L-infty} the constant $M_1$ in \eqref{grad-est} can be evidently chosen to be independent of $\sigma$. 

Finally, we rewrite the problem \eqref{VP1-sigma}--\eqref{VP2-sigma} in the form
\begin{align*}
&-\tilde{a}^{ij}D_iD_ju + \sigma\tilde{b}^iD_iu =  -\sigma F
&\text{a.e. in}\ \Omega,\\
&-\tilde{\alpha}^{ij}d_id_ju +
\sigma\tilde\beta^{*i}d_iu + \sigma\tilde\beta_0\partial_{\mathbf{n}}u+u =
-\sigma G
&\text{a.e. on}\ \partial\Omega,
\end{align*}
with
\begin{align*}
F = \tilde{a}\big(\mu|Du|^2+\Phi\big),\qquad
G = \tilde\alpha\big(\mu|du|^2+\Theta\big)-u,
\end{align*}
and conclude by \eqref{grad-est} that $\|F\|_{n,\Omega}$ and $\|G\|_{n-1,\partial\Omega}$ are bounded uniformly with respect to $\sigma$ and $u$. Thus, \eqref{2.3-AN95'} implies
$$
\|u\|_{V_{n,n-1}(\Omega)}\le C
$$
with $C$ independent of $u$ and $\sigma$, which immediately implies the desired \textit{a~priori} estimate \eqref{AA}.
Application of the Leray--Schauder fixed point theorem completes the proof of 
Theorem~\ref{quasilinear-existence}. 
\end{proof}

Theorem~\ref{quasilinear-existence} is of conditional
 type. It reduces the solvability of the quasilinear Venttsel problem \eqref{1.1-AN95}--\eqref{1.2-AN95} to the \textit{a~priori} estimate \eqref{L-infty} for {\it any} solution $u\in V_{n,n-1}(\Omega)$ of the problem \eqref{VP1-sigma}--\eqref{VP2-sigma}, with $M_0$ independent of $u$ and $\sigma\in[0,1]$. We provide now a simple sufficient condition ensuring the validity of \eqref{L-infty}.

\begin{lemma}\label{M0}
Let $\partial\Omega\in W^2_n$. Assume that conditions \eqref{A1}, \eqref{V} and \eqref{V1} are satisfied, together with
 \eqref{A4} and \eqref{V4} where
$$
b, \Phi \in L^n(\Omega);\quad
\beta,\Theta \in L^{n-1}(\partial\Omega).
$$
Finally, suppose that for $|z|\ge z_0$ the functions $a(x,z,p)$ and $\alpha(x,z,p^*)$ with $p^*\perp {\bf n}(x)$ are weakly differentiable with respect to $z$ and
$$
a_z(x,z,p)\ge \theta_0 \Phi(x);\qquad
\alpha_z(x,z,p^*)\ge \theta_0 \max\big\{1,\Theta(x)\big\},\quad \theta_0={\text{\rm const}}>0.
$$

Then any solution $u\in V_{n,n-1}(\Omega)$ of  \eqref{VP1-sigma}--\eqref{VP2-sigma} satisfies the bound \eqref{L-infty} with $M_0=z_0+\theta_0^{-1}\eta(z_0)$.
\end{lemma}

\begin{proof}
We will get the estimate $u\le M_0$, the proof of $u\ge -M_0$ is similar.

Suppose that the set $\Omega^+=\{u>z_0\}$ is non-empty. Then, using \eqref{A4} and the lower bound for $a_z$, we can write a.e. in $\Omega\cap\Omega^+$ that
\begin{align*}
a(x,u,Du) &\ =a(x,z_0,Du)+\int_0^1 a_z\big(x,tu(x)+(1-t)z_0,Du(x)\big)(u-z_0)\,dt \\
&\ \ge -\eta(z_0)\Big(\mu |Du(x)|^2+b(x)|Du(x)|+\Phi(x)\Big)+\theta_0 \Phi(x)(u-z_0).
\end{align*}
Therefore, \eqref{VP1-sigma} yields
\begin{align}
\label{m1}
\mathcal{L}u(x) :=&\ -\tilde{a}^{ij}(x)D_iD_ju +\sigma\eta(z_0) b^i(x)D_iu +\sigma\theta_0\Phi(x)u\\
\nonumber
\le&\ \big(\eta(z_0)+\theta_0z_0\big)\,\sigma\Phi(x)\\
\nonumber
=&\ \mathcal{L}\big(\theta_0^{-1}\eta(z_0)+z_0\big),
\end{align}
where 
$$
\tilde{a}^{ij}(x):={a}^{ij}\big(x,u(x)\big),\qquad b^i(x):=-D_iu(x)\Big(\mu+\dfrac{b(x)}{|Du(x)|}\Big)\in L^n(\Omega).
$$

Further on, assuming without loss of generality that $\Theta\ge 1$, we write $\alpha(x,u,Du)=\alpha(x,u,du)+\tilde\beta_0(x)\partial_{\mathbf{n}}u$ as in the proof of Theorem~\ref{Du-est} with $\tilde\beta_0(x)$ given by \eqref{beta0u}. Using \eqref{V4} and the lower bound for $\alpha_z$, we obtain a.e. on $\partial\Omega\cap\Omega^+$ that
\begin{align*}
\alpha(x,u,Du) =&\ \alpha(x,z_0,du)+\int_0^1 \alpha_z\big(x,tu(x)+(1-t)z_0,du(x)\big)(u-z_0)\,dt \\
&\ +\tilde\beta_0(x)\partial_{\mathbf{n}}u \\
\ge&\  -\eta(z_0)\Big(\mu |du(x)|^2+\beta(x)|du(x)|+\Theta(x)\Big)+\theta_0 \Theta(x)(u-z_0)\\
&\ +\tilde\beta_0(x)\partial_{\mathbf{n}}u.
\end{align*}
This way, \eqref{VP2-sigma} implies
\begin{align}
\label{m2}
\mathcal{B}u(x)  :=&\ -\tilde{\alpha}^{ij}(x)d_id_ju+\sigma\eta(z_0) \beta^{*i}(x)d_iu+\sigma \tilde\beta_0(x)\partial_{\mathbf{n}}u\\
\nonumber
&\  +\big(1-\sigma+\sigma\theta_0\Theta(x)\big)u\\
\nonumber
\le&\ \big(\eta(z_0)+\theta_0z_0\big)\,\sigma\Theta(x)\le \mathcal{B}\big(\theta_0^{-1}\eta(z_0)+z_0\big)
\end{align}
with 
$$
\tilde{\alpha}^{ij}(x):={\alpha}^{ij}\big(x,u(x)\big), \quad\beta^{*i}(x):=-d_iu(x)\Big(\mu+\dfrac{\beta(x)}{|du(x)|}\Big)\in L^{n-1}(\partial\Omega).
$$
It is to be noted  that $\tilde{\beta}_0\ge0$ as it follows from \eqref{V}, while $1-\sigma+\sigma\theta_0\Theta(x)\ge \min\{1,\theta_0\}>0$.

Setting $M_0:=\theta_0^{-1}\eta(z_0)+z_0$, we have from \eqref{m1} and \eqref{m2} that
$$
\CL(u-M_0)\le 0\quad \text{a.e. in}\ \Omega\cap\Omega^+,\qquad
\CB(u- M_0)\le 0\quad \text{a.e. on}\ \partial\Omega\cap\Omega^+.
$$
To get the claim, we suppose the contrary, that is, let $\max_{\overline\Omega}\,(u-M_0)>0$. Then the function $u(x)-M_0$ achieves its maximum at a point $x^0 \in \Omega^+$ and the Aleksandrov--Bakel'man maximum principle implies that $x^0 \in \partial\Omega\cap\Omega^+$. This leads to a  contradiction as in Theorem~\ref{global-max-principle} and that completes the proof.
\end{proof}

It is to be noted that the monotonicity hypotheses on  $a(x,z,p)$ and $\alpha(x,z,p^*)$ in Lemma~\ref{M0} could be replaced with suitable sign-conditions of these terms with respect to $z.$

\section{Concluding remarks and open problems}
\label{sec5}

\textbf{1.}
In Sections~\ref{sec3} and \ref{sec4} we dealt
with the multidimensional case $n\geq 3.$ 
If $n=2$, then the problem is simpler because the boundary equation is essentially an ordinary differential equation. So, in this case we can allow $q=1$ in \eqref{pq-cond}. Moreover, for both linear and quasilinear problems, the principal coefficient in the \textit{boundary} operator can be merely measurable. All the results obtained in Sections~\ref{sec3} and \ref{sec4} remain valid in the 2D case.

\smallskip

\textbf{2.}
The machinery  developed in Sections~\ref{sec3} and \ref{sec4} runs without essential changes also for \textit{two-phase} linear and quasilinear Venttsel problems with discontinuous coefficients (cf. \cite{AN00a,AN01}).

\smallskip

\textbf{3.}
Modulo technical details, the results of Section~\ref{sec3}  can be extended to operators with \textit{partially VMO} principal coefficients in the spirit of \cite{Kr07}. Moreover, since all statements in Section~\ref{sec4} use the \textit{VMO}-hypothesis only via the linear theory, assumptions \eqref{A2} and \eqref{V2}  can be weakened in a similar manner. 

\smallskip

\textbf{4.}
The integrability assumptions \eqref{L3}, \eqref{L4}, \eqref{B3}--\eqref{B5} on the 
lower-order coefficients of the linear operators $\mathcal{L}$ and $\mathcal{B}$ are \textit{sharp} for the validity of the results obtained in Section~\ref{sec3}. The technique employed in the proof of Theorem~\ref{apriori-estimate} can be successfully adopted to Dirichlet or  oblique derivative problems for elliptic operator with \textit{VMO} coefficients in order to generalize the results in \cite{CFL91,CFL93} and \cite{FaP96a,MP98}. 

\smallskip

\textbf{5.}
The assumptions \eqref{A4} and \eqref{V4} allow not only quadratic gradient growth but also linear one with an \textit{unbounded} coefficient. The approach used in the proof of Theorem~\ref{Du-est}, applied to quasilinear Dirichlet or regular oblique derivative problem, could give generalization of the results from \cite{P95} and \cite{FaP96b}.

\subsection*{Open problems and directions for  further research}
The results in Section~\ref{sec4} regard quasilinear elliptic operators and quasilinear Venttsel boundary conditions with principal coefficients depending on $x$ and $u$ only. In order to get statements at the same level of generality as in classical papers of O.A.~Ladyzhenskaya and N.N.~Ural'tseva (see the survey \cite{LU86}), it is necessary to allow dependence also   on $Du$ and $du$, respectively, in the principal coefficients, and this is a fascinating open problem.

\medskip

Another interesting problem is to  extend the results of \cite{AN02} to the case of discontinuous principle coefficients.
If the support of the Venttsel condition in the two-phase  problem is the interface meeting  transversally the exterior boundary of the medium, then both parts of the medium are domains with smooth closed edges. 
This requires to study the problem in weighted
Sobolev spaces where the weight is a power of the distance from a point to the edge (see  \cite{AN02} for more details). 

\medskip

Very difficult open problems appear in the case where the  principal part of the boundary operator 
is \textit{positively semidefinite.} 
There occur three different types of degeneracy:

\smallskip

\textit{1) Partial}, when the matrix $\{\alpha^{ij}(x)\}$ is not vanishing but has zero as an eigenvalue. There are some results due to 
Luo and Trudinger~\cite{LuT91}
regarding the linear case, but the quasilinear one is completely open.

\smallskip

\textit{2) Complete}, where the set
$\mathcal{D}=\big\{x\in\partial\Omega\colon\ \alpha^{ij}(x)=0\big\}$ is non-empty, but $\beta_0\neq0$ there. Now the boundary value problem is of oblique derivative type on $\mathcal{D}$, and 
some satisfactory results are available in \cite{AN97,AN98,N04}.  We refer the reader also to \cite[Chapter~6]{Pan00} where similar problems are treated with the machinery of pseudo-differential operators and H\"ormander vector fields.  It is a challenging task to overcome the continuity of the principal coefficients in that case.

\smallskip

\textit{3) Over-degeneration}, when $\mathcal{D}\neq\emptyset$ and $\mathcal{D}_0=\big\{x\in\mathcal{D}\colon\ \beta_0(x)=0\big\}$ is \textit{non-empty}. Now the Venttsel boundary condition is prescribed in terms only of \textit{tangential} derivatives, 
${\beta^*}^i(x)d_iu+\gamma(x)u,$
and the corresponding boudary value problem is no more regular (see \cite{EK69,PP97,Pan00,P05,P06,P08}).
 The properties of the problem, known as Poincar\'e problem,  depend essentially on the behaviour of the vector field $\boldsymbol{\beta^*}$ near the set $\mathcal{D}_0,$ and new effects appear such as loss of smoothness, loss of Fredholmness, etc. Nothing is known, instead, for the quasilinear Poincar\'e problem.

\subsection*{Acknowledgements} 
This work was partly supported by RUDN University Program 5-100 (D.A.), by RFBR grant 18-01-00472 (D.A. and A.N.), by the grant AP05130222 of Kazakhstan Ministry of Education and Science (A.N.). The work of D.K.P. was supported by the Italian Ministry of Education, University and Research under the Programme “Department of Excellence” Legge 232/2016 (Grant No. CUP-D94I18000260001). The research of L.G.S. was partially supported by the Project GNAMPA 2020 “Elliptic operators with unbounded and singular coefficients on weighted $L^p$ spaces”.

A part of this work was done during the visits of D.A. and A.N. to Politecnico di Bari in 2018, partially supported by Visiting Professorship Program~2018 of Politecnico di Bari, and by the St. Petersburg University (project 34827971), respectively.\smallskip

The authors are indebted to the referees for the valuable remarks.


\begin{thebibliography}{99}
\bibitem{A92}
{\sc P. Acquistapace},
{\em On {\textit{BMO}} regularity for linear elliptic systems},
Ann. Mat. Pura Appl. (4), 161 (1992), pp.~231--269.

\bibitem{Al63}
{\sc A. D. Aleksandrov},
{\em Uniqueness conditions and bounds for the solution of the {D}irichlet problem},
Vestnik Leningrad. Univ. Mat. Meh. Astron., 18 (1963), no. 3,  pp.~5--29.

\bibitem{AC78}
{\sc H. Amann and M. Crandall},
{\em On some existence theorems for semi-linear elliptic equations},
Indiana Univ. Math. J., 27 (1978), pp.~779--790.


\bibitem{AN95a}
{\sc D. E. Apushkinskaya and A. I. Nazarov},
{\em H\"{o}lder estimates of solutions to initial-boundary value problems for parabolic equations of nondivergent form with {W}entzel boundary condition}, in
Nonlinear Evolution Equations, 
Amer. Math. Soc. Transl. Ser. 2,  164,
 Amer. Math. Soc., Providence, RI, 1995, pp. 1--13.
 
\bibitem{AN95}
{\sc D. E. Apushkinskaya and A. I. Nazarov},
{\em An initial-boundary value problem with a {V}enttsel \ boundary condition for parabolic equations not in divergence form},
St. Petersburg Math. J., 6 (1995), no. 6, pp.~1127--1149.

\bibitem{AN97}
{\sc D. E. Apushkinskaya and A. I. Nazarov},
{\em H\"{o}lder estimates for solutions to the degenerate boundary-value {V}enttsel problem for parabolic and elliptic equations of nondivergence type},
Probl. Mat. Anal., 17 (1997), pp.~3--19 [Russian]; 
English transl.: J. Math. Sci. (N.Y.), 97 (1999), no. 4, pp.~4177--4188.

\bibitem{AN98}
{\sc D. E. Apushkinskaya and A. I. Nazarov},
{\em Estimates for the gradient of solutions to stationary degenerate {V}enttsel'  problems},
Probl. Mat. Anal., 18 (1998), pp.~43--68 [Russian]; 
English transl.: J. Math. Sci. (N.Y.), 98 (2000), no. 6, pp.~654--673.

\bibitem{AN00}
{\sc D. E. Apushkinskaya and A. I. Nazarov},
{\em A survey of results on nonlinear {V}enttsel problems},
Appl. Math., 45 (2000), no. 1, pp.~69--80.

\bibitem{AN00a}
{\sc D. E. Apushkinskaya and A. I. Nazarov},
{\em Quasilinear two-phase {V}enttsel' problems},
Zap. Nauchn. Sem. S.-Peterburg. Otdel. Mat. Inst. Steklov. (POMI), 271 (2000), pp.~11--38 [Russian]; English transl.: J. Math. Sci. (N.Y.), 115 (2003), no. 6, pp.~2704--2719.

\bibitem{AN01}
{\sc D. E. Apushkinskaya and A. I. Nazarov},
{\em Linear two-phase {V}enttsel problems},
Ark. Mat., 39 (2001), no. 2, pp.~201--222.

\bibitem{AN02}
{\sc D. E. Apushkinskaya and A. I. Nazarov},
{\em Quasilinear elliptic two-phase {V}enttsel's problems in the transversal case},
J. Math. Sci. (N.Y.), 112 (2002), no. 1, pp.~3927--3943.

\bibitem{BIN75}
{\sc O. V. Besov,  V. P. Il'in and S. M. Nikol'ski\u\i},
\textit{\em Integral Representations of Functions, and Embedding Theorems},
Nauka, Moscow, 1975 [Russian]; 
English transl.: V.H. Winston and Sons, Washington, DC; Halsted press, New York-Toronto, Ont.-London, I, 1978, II, 1979.

\bibitem{CaMe71}
{\sc J. R. Cannon and G. H. Meyer},
{\em On diffusion in a fractured medium},
SIAM J. Appl. Math., 3 (1971), pp.~434--448.

\bibitem{CFL91}
{\sc F. Chiarenza, M. Frasca and P. Longo},
{\em Interior {$W^{2,p}$} estimates for nondivergence elliptic equations with discontinuous coefficients}, Ric. Mat., 40 (1991), no. 1,
pp.~149--168.

\bibitem{CFL93}
{\sc F. Chiarenza, M. Frasca and P. Longo},
{\em {$W^{2,p}$}-solvability of the {D}irichlet problem for nondivergence elliptic equations with \textit{VMO} coefficients},
Trans. Amer. Math. Soc., 336 (1993), no. 2, pp.~841--853.

\bibitem{CF3GOR09}
{\sc G. M. Coclite, A. Favini,  C. G. Gal,  G. R. Goldstein,   J. A. Goldstein, E. Obrecht and S. Romanelli},
{\em The role of {W}entzell boundary conditions in linear and nonlinear analysis}, in
Advances in Nonlinear Analysis: Theory Methods and Applications,
Math. Probl. Eng. Aerosp. Sci. 3,
Camb. Sci. Publ., Cambridge, 2009, pp.~277--289.

\bibitem{CLNV19}
{\sc S. Creo, M. R. Lancia, A. I. Nazarov and P. Vernole}, 
{\em On two-dimensional nonlocal {V}enttsel' problems in piecewise smooth domains},
Discrete Contin. Dyn. Syst. Ser. S, 12 (2019), no. 1, pp.~57--64.

\bibitem{EK69}
{\sc Y. V. Egorov and V. A. Kondrat’ev}, 
{\em The oblique derivative problem}, 
Math. USSR Sbornik, 7 (1969), no. 1, pp.~139--169.

\bibitem{FaP96a}
{\sc G. Di Fazio and D. K. Palagachev},
{\em Oblique derivative problem for elliptic equations in non-divergence form with \textit{VMO} coefficients},
Comment. Math. Univ. Carolin., 37 (1996), no. 3, pp.~537--556.

\bibitem{FaP96b}
{\sc G. Di Fazio and D. K. Palagachev},	
{\em Oblique derivative problem for quasilinear elliptic equations with \textit{VMO} coefficients},
Bull. Aust. Math. Soc., 53 (1996), no. 3, pp.~501--513.

\bibitem{GGM08}
{\sc C. G. Gal, M. Grasselli and A. Miranville},
{\em Nonisothermal {A}llen--{C}ahn equations with coupled dynamic boundary conditions}, in
Nonlinear Phenomena with Energy Dissipation, 
GAKUTO Internat. Ser. Math. Sci. Appl. 29, Gakk\"{o}tosho, Tokyo, 2008, pp.~117--139.

\bibitem{GM09}
{\sc C. G. Gal and A. Miranville},
{\em Uniform global attractors for non-isothermal viscous and non-viscous {C}ahn-{H}illiard equations with dynamic boundary conditions},
Nonlinear Anal. Real World Appl., 10 (2009), no. 3, pp.~1738--1766.

\bibitem{GaSk01}
{\sc E. I. Galakhov and A. L. Skubachevski\u{\i}},
{\em On {F}eller semigroups generated by elliptic operators with integro-differential boundary conditions},
J. Differential Equations, 176 (2001), no. 2, pp.~315--355.

\bibitem{GT01}
{\sc D. Gilbarg  and N. S. Trudinger},
{\em Elliptic Partial Differential Equations of Second Order},
Reprint of the 1998 edition, Classics in Mathematics, Springer--Verlag, Berlin, 2001.

\bibitem{Gol06}
{\sc G.R. Goldstein},
{\em Derivation and physical interpretation of general boundary conditions}, Adv. Differential Equations, 11 (2006), no. 4, pp.~457--480.

\bibitem{HS74}
{\sc P. H. Hung and E. S\'{a}nchez--Palencia},
{\em Ph\'{e}nom\`enes de transmission \`a travers des couches minces de conductivit\'{e} \'{e}lev\'{e}e}, J. Math. Anal. Appl., 47 (1974), pp.~284--309.

\bibitem{IW89}
{\sc N. Ikeda and S. Watanabe},
{\em Stochastic Differential Equations and Diffusion Processes},
{North-Holland Math. Library}, Vol. 24, 
{North-Holland Publishing Co., Amsterdam; Kodansha, Ltd., Tokyo}, 1989.

\bibitem{JN61}
{\sc F. John and L. Nirenberg},
{\em On functions of bounded mean oscillation},
Comm. Pure Appl. Math., 14 (1961), pp.~415--426.

\bibitem{KR58}
{\sc M. A. Krasnosel'ski\u{\i}  and Ja. B. Ruticki\u{\i}},
{\em Convex Functions and {O}rlicz Spaces},
Problems Contemp. Math., Gosudarstv. Izdat. Fiz.-Mat. Lit., Moscow 1958 [Russian];
English transl.: P. Noordhoff Ltd., Groningen, 1961.

\bibitem{Kr07}
{\sc N. V. Krylov},
{\em Second-order elliptic equations with variably partially \textit{VMO} coefficients},
J. Funct. Anal., 257 (2009), no. 6, pp.~1695--1712.


\bibitem{LU68}
{\sc O. A. Ladyzhenskaya and N. N. Ural'tseva},
{\em Linear and Quasilinear Elliptic Equations},
{Academic Press, New York-London}, 1968.

\bibitem{LU86}
{\sc O. A. Ladyzhenskaya and N. N. Ural'tseva},
{\em A survey of results on the solvability of boundary value
problems for uniformly elliptic and parabolic second-order
quasilinear equations having unbounded singularities}, Russian Math. Surveys, 41 (1986), no.5, pp.~1--31.

\bibitem{La02}
{\sc M. R. Lancia},
{\em A transmission problem with a fractal interface}, Z. Anal. Anwend., 21 (2002), no. 1, pp.~113--133.

\bibitem{La03}
{\sc M.R. Lancia},
{\em Second order transmission problems across a fractal surface},
Rend. Accad. Naz. Sci. XL Mem. Mat. Appl. (5), 27 (2003), pp.~191--213.

\bibitem{LM68}
{\sc J.-L. Lions and E. Magenes},
{\em Probl\`emes aux Limites Non Homog\`enes et Applications}, 
{Travaux et Recherches Math\'{e}matiques,} 1, No. 17, 	{Dunod, Paris}, 1968.

\bibitem{LN98}
{\sc V. V. Luk'yanov and A. I. Nazarov},
{\em Solution of the {V}enttsel problem for the {L}aplace and the {H}elmholtz equations by means of iterated potentials}, Zap. Nauchn. Sem. S.-Peterburg. Otdel. Mat. Inst. Steklov. (POMI), 250 (1998), pp.~203--218 [Russian]; 
English transl.: J. Math. Sci. (N.Y.), 102 (2000), no. 4, pp.~4265--4274.

\bibitem{Lu91}
{\sc Y. Luo},
{\em On the quasilinear elliptic {V}enttsel' boundary value problem},
Nonlinear Anal., 16 (1991), no. 9, pp.~761--769.

\bibitem{LuT91}
{\sc Y. Luo and N. S. Trudinger},
{\em Linear second order elliptic equations with {V}enttsel' boundary conditions},
Proc. Roy. Soc. Edinburgh Sect. A, 118 (1991), no. 3-4, pp.~193--207.

\bibitem{MP98}
{\sc A. Maugeri and D. K. Palagachev},
{\em Boundary value problem with an oblique derivative for uniformly elliptic operators with discontinuous coefficients},
Forum Math., 10 (1998), no. 4, pp.~393--405.

\bibitem{MPS00}
{\sc A. Maugeri, D. K. Palagachev and L. G. Softova},
{\em Elliptic and Parabolic Equations with Discontinuous Coefficients},
Math. Res. 109, {Wiley-VCH Verlag Berlin GmbH, Berlin}, 2000.

\bibitem{MNaPr86}
{\sc N. F. Morozov, S. A. Nazarov  and A. V. Proskura}, 
{\em Boundary-value problems of elasticity theory for plane domains with thin bordering}, in
Mekh. {D}eform. {T}el, 
{Nauka, Moscow}, 1986, pp. 82--93 [Russian].

\bibitem{N04}
{\sc A. I. Nazarov},
{\em The degenerate {V}enttsel' problem for elliptic equations},
Zap. Nauchn. Sem. S.-Peterburg. Otdel. Mat. Inst. Steklov. (POMI), 310 (2004), pp.~83--97 [Russian]; English transl.:
J. Math. Sci. (N.Y.), 132 (2006), no. 3, pp.~295--303.

\bibitem{N05}
{\sc A. I. Nazarov},
{\em The maximum principle of A.D. Aleksandrov},
Sovrem. Mat. Prilozh., 29 (2005), pp.~129--145 [Russian]; English transl.: J. Math. Sci. (N.Y.), 142 (2007), no. 3, pp.~2154--2171.

\bibitem{NP15}
{\sc A. I. Nazarov and A. A. Paletskikh},
{\em On the {H}\"{o}lder continuity of solutions of the {V}enttsel' eliptic problem},
Dokl. Math., 92 (2015),
pp.~747--751.

\bibitem{NaPi93}
{\sc S. Nazarov and K. Pileckas},
{\em On noncompact free boundary problems for the plane stationary {N}avier--{S}tokes equations}, J. Reine Angew. Math., 438 (1993), pp.~103--141.

\bibitem{P95}
{\sc D. K. Palagachev},
{\em Quasilinear elliptic equations with \textit{VMO} coefficients},
Trans. Amer. Math. Soc., 347 (1995), no. 7, pp.~2481--2493.


\bibitem{P05}
{\sc D. K. Palagachev},
{\em The Poincar\'e problem in $L^p$-Sobolev spaces, I:~Codimension one degeneracy},
J. Funct. Anal., 229 (2005), no. 1, pp.~121--142.

\bibitem{P06}
{\sc D. K. Palagachev},
{\em Neutral Poincar\'e problem in $L^p$-Sobolev spaces: regularity and Fredholmness},
Int. Math. Res. Not., 2006, (2006), Art. ID 87540, 31 pp.

\bibitem{P08}
{\sc D. K. Palagachev},
{\em The Poincar\'e problem in $L^p$-Sobolev spaces, II:~Full dimension degeneracy},
Comm. Partial Differential Equations, 33 (2008), no. 1-3, pp.~209--234.


\bibitem{Pan00}
{\sc B.~Paneah},
{\em The Oblique Derivative Problem. The Poincar\'e Problem},
{Math. Topics 17, Wiley-VCH Verlag Berlin GmbH, Berlin}, 2000.

\bibitem{PP97}
{\sc P. R. Popivanov and D. K. Palagachev},
{\em The Degenerate Oblique Derivative Problem for Elliptic and Parabolic Equations},
Math. Res. 93, {Akademie Verlag, Berlin}, 1997.

\bibitem{Pe98}
{\sc O. M. Penkin},
{\em On the maximum principle for an elliptic equation on stratified sets},
Difer. Uravn., 34 (1998), no. 10, pp.~1433--1434 [Russian]; 
English transl.: Diff. Eq., 34 (1998), no. 8, pp.~1111--1117.

\bibitem{R86}
{\sc R.E. Raspe},
{\em {B}aron {M}unchausen's Narrative of his Marvellous Travels and Campaigns in {R}ussia}, {Oxford}, 1786.


\bibitem{Sar75}
{\sc D. Sarason},
{\em Functions of vanishing mean oscillation},
Trans. Amer. Math. Soc., 207 (1975), pp.~391--405.

\bibitem{SV95}
{\sc T. B. A. Senior  and J. L. Volakis},
{\em Approximate {B}oundary {C}onditions in {E}lectromagnetics},
{IEE Electromagnetics Wave Series} 41, 
{IEE, London}, 1995.

\bibitem{Sh80}
{\sc M. Shinbrot},
{\em Water waves over periodic bottoms in three dimensions},
J. Inst. Math. Appl., 25 (1980), no. 4, pp.~367--385.

\bibitem{Shi99}
{\sc A. N. Shiryaev},
{\em Essentials of Stochastic Finance},
{Advanced Series on Statistical Science \& Applied Probability} 3, 
{World Scientific Publishing Co., Inc., River Edge, NJ}, 1999.

\bibitem{Tri78}
{\sc H. Triebel},
{\em Interpolation Theory, Function Spaces, Differential Operators},
{North-Holland Mathematical Library} 18,
{North-Holland Publishing Co., Amsterdam-New York}, 1978.

\bibitem{V59}
{\sc A. D. Venttsel},
{\em On boundary conditions for multidimensional diffusion processes}, Teor. Veroyatnost. i Primenen., 4 (1959), pp.~172--185 
[Russian]; English transl.: Theor. Probability Appl., 4 (1960) pp.~164--177.

\bibitem{W79}
{\sc S. Watanabe},
{\em Construction of diffusion processes with {W}entzell's boundary conditions by means of {P}oisson point processes of {B}rownian excursions}, in
Probability Theory ({P}apers, {VII}th {S}emester, {S}tefan {B}anach 
{I}nternat. {M}ath. {C}enter, {W}arsaw, 1976), Vol. 5,
{Banach Center Publ.}, {PWN, Warsaw}, 1979.
\end{thebibliography}



\end{document}